\documentclass[12pt]{article}
\usepackage{amssymb}
\usepackage{amsmath,amsthm}
\usepackage[latin1]{inputenc}
\usepackage{graphicx}
\usepackage{hyperref}
\usepackage{enumerate}
\usepackage{tikz}
\usepackage{tkz-graph}
\usepackage{mathrsfs}
\usepackage{verbatim}

\hypersetup{colorlinks=true, linkcolor=blue, citecolor=blue, urlcolor=blue}

\newtheorem{remark}{Remark}

 \newtheorem{definition}{Definition}

\newtheorem{lemma}[remark]{Lemma}
\newtheorem{theorem}[remark]{Theorem}
\newtheorem{proposition}[remark]{Proposition}
\newtheorem{corollary}[remark]{Corollary}

  \newtheorem{problem}[definition]{Problem}

\title{On the weak Roman domination number of lexicographic product graphs}

\author{Magdalena Valveny, Hebert P\'{e}rez-Ros\'{e}s, Juan A. Rodr\'{\i}guez-Vel\'{a}zquez
 \\
\\
{\small Departament d'Enginyeria Inform\`atica i Matem\`atiques
}\\
{\small Universitat Rovira i Virgili,  Av. Pa\"{\i}sos Catalans 26,
43007 Tarragona, Spain.} \\{\small magdalena.valveny@estudiants.urv.cat, hebert.perez@urv.cat, juanalberto.rodriguez\@@urv.cat}
}

\begin{document}
\maketitle

\begin{abstract}
A vertex $v$ of a graph $G=(V,E)$  is said to be undefended with respect to a function $f: V \longrightarrow \{0,1,2\}$ if  $f(v)=0$ and $f(u)=0$ for every vertex $u$ adjacent to $v$. We call the function $f$ a weak  Roman dominating function  if for every $v$ such that $f(v)=0$ there exists a vertex $u$ adjacent to $v$ such that $f(u)\in  \{1,2\}$ and the function $f': V \longrightarrow \{0,1,2\}$ defined by $f'(v)=1$, $f'(u)=f(u)-1$ and $f'(z)=f(z)$ for every $z\in V \setminus\{u,v\}$, has no undefended vertices. 
The weight of $f$ is $w(f)=\sum_{v\in V(G) }f(v)$.
The  weak Roman domination number of a graph $G$, denoted by $\gamma_r(G)$, is  the minimum weight among all weak Roman dominating functions on $G$. Henning and Hedetniemi [Discrete Math. 266 (2003) 239-251] showed that the problem of computing $\gamma_r(G)$ is NP-Hard, even when restricted to bipartite or chordal graphs. This suggests finding
$\gamma_r(G)$ for special classes of graphs or obtaining good bounds on this invariant. 
In this  article, we obtain closed formulae and tight bounds for the weak  Roman domination number of lexicographic product graphs in terms of invariants of the factor graphs involved in the product.  
\end{abstract}



\section{Introduction}

 Cockaine et al. \cite{Cockayne2004} defined a {\it Roman dominating  function}  (RDF) on a graph $G$ to be a function $f: V(G)\longrightarrow \{0,1,2\}$ satisfying the condition that every vertex $u$ for which $f(u)=0$ is adjacent to at least one vertex $v$ for which $f(v)=2$. The \textit{weight} of $f$ is $w(f)=\sum_{v\in V(G)}f(v)$, and for  $X\subseteq V(G)$ we define the weight of $X$  to be $f(X)= \sum_{v\in X}f(v)$. The \textit{Roman domination number}, denoted by   $\gamma_R(G)$, is the minimum weight among all Roman dominating functions on $G$, \textit{i.e.}, $$\gamma_R(G)=\min\{w(f):\, f \text{ is a RDF on } G\}.$$ The Roman domination theory was motivated by an article by Ian Stewart entitled ``Defend the Roman Empire!" \cite{Stewart1999}. Each vertex in our graph represents a \textit{location} in the Roman Empire and the value $f(v)$ corresponds to the number of \textit{legions}  stationed in $v$. A location $v$ is \textit{unsecured} if no legions are stationed there (\textit{i.e.}, $f(v)=0$) and \textit{secured} otherwise (\textit{i.e.}, $f(v)\in \{1,2\}$). An unsecured location $u$ can be secured by sending a legion to $u$  from an adjacent
location $v$. But a legion cannot be sent from a location $v$ if doing so leaves that location unsecured
(\textit{i.e.} if $f(v) = 1$). Thus, two legions must be stationed at a location ($f(v) = 2$) before
one of the legions can be sent to an adjacent location. A RDF of weight $\gamma_R(G)$ corresponds to such an optimal assignment of legions to locations.

Henning and Hedetniemi \cite{MR1991720} explored the potential of saving sustantial costs of maintaining legions, while still defending the Roman Empire (from a single attack). They proposed the use of \textit{weak Roman dominating functions} (WRDF) as follows. Using the terminology introduced earlier,
they defined a location to be \textit{undefended} if the location and every location adjacent to it are unsecured (\textit{i.e}., have no legion stationed there). Since an undefended location
is vulnerable to an attack, we require that every unsecured location be adjacent to a secure location in such a way that the movement of a legion from the secure location to the unsecured location does not create an undefended location. Hence every unsecured location can be defended without creating an undefended location. Such a placement of legions corresponds to a WRDF and a minimum such placement of legions corresponds to a minimum WRDF. More formally, a vertex $v\in V(G)$ is \textit{undefended} with respect to a function 
$f: V(G)\longrightarrow \{0,1,2\}$ if  $f(v)=0$ and $f(u)=0$ for every vertex $u$ adjacent to $v$. We call the function $f$ a WRDF if for every $v$ such that $f(v)=0$ there exists a vertex $u$ adjacent to $v$ such that $f(u)\in  \{1,2\}$ and the function $f': V(G)\longrightarrow \{0,1,2\}$ defined by $f'(v)=1$, $f'(u)=f(u)-1$ and $f'(z)=f(z)$ for every $z\in V(G)\setminus\{u,v\}$, has no undefended vertices. 

Let $f$ be a WRDF on $G$ and let $V_0$, $V_1$ and $V_2$  be the subsets of vertices assigned  the values $0$, $1$, and $2$, respectively, under $f$. Notice that there is a one-to-one correspondence between the set  of weak Roman dominating functions  $f$ and the set of ordered partitions  $(V_0,V_1,V_2)$ of $V(G)$. Thus, in order to specify the partition of $V(G)$ associated to $f$,  the function $f$ will be denoted by  $f(V_0,V_1,V_2)$.

The \textit{weak Roman domination number}, denoted by  $\gamma_r(G)$, is the minimum weight among all weak Roman dominating functions on $G$, \textit{i.e.}, $$\gamma_r(G)=\min\{w(f):\, f \text{ is a WRDF on } G\}.$$
A WRDF of weight  $\gamma_r(G)$ is called a $\gamma_r(G)$-function.  For instance, for the tree shown in Figure \ref{FigWeakRoman}  a $\gamma_r(G)$-function can place $2$ legions in the vertex of degree three and one legion in the other black-coloured vertex. Notice that $\gamma_r(G)=3<4=\gamma_R(G)$.

\begin{figure}[htb]
\begin{center}
\begin{tikzpicture}
[line width=1pt, scale=0.8]

\coordinate (V1) at (1,0);
\coordinate (V2) at (2,0);
\coordinate (V3) at (3,0);
\coordinate (V4) at (4,0);
\coordinate (V5) at (1,-1);
\coordinate (V6) at (1,1);

\draw[black!40]  (V1)--(V2)--(V3)--(V4);
\draw[black!40]  (V1)--(V5);
\draw[black!40]  (V1)--(V6);

\foreach \number in {1,...,6}{
\filldraw[gray]  (V\number) circle (0.08cm);
}
\foreach \number in {1,3}{
\filldraw[black]  (V\number) circle (0.08cm);
}

\node [left] at (1,0) {$2$};
\node [above] at (3,0) {$1$};
\end{tikzpicture}
\hspace{1cm}
\begin{tikzpicture}
[line width=1pt, scale=0.8]

\coordinate (V1) at (1,0);
\coordinate (V2) at (2,0);
\coordinate (V3) at (3,0);
\coordinate (V4) at (4,0);
\coordinate (V5) at (1,-1);
\coordinate (V6) at (1,1);

\draw[black!40]  (V1)--(V2)--(V3)--(V4);
\draw[black!40]  (V1)--(V5);
\draw[black!40]  (V1)--(V6);

\foreach \number in {1,...,6}{
\filldraw[gray]  (V\number) circle (0.08cm);
}
\foreach \number in {1,4}{
\filldraw[black]  (V\number) circle (0.08cm);
}

\node [left] at (1,0) {$2$};
\node [above] at (4,0) {$1$};
\end{tikzpicture}
\end{center}
\vspace{-0,5cm}
\caption{Two placements of legions which correspond  to  two different weak Roman dominating functions on the same tree.}
\label{FigWeakRoman} 
\end{figure}
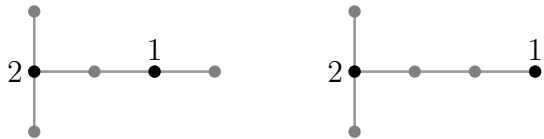

It was shown in \cite{MR1991720} that the problem of computing $\gamma_r(G)$ is NP-hard, even when restricted to bipartite or chordal graphs. This suggests finding
the weak Roman domination number for special classes of graphs or obtaining good bounds on this invariant. In this paper we develop the theory of weak Roman domination in lexicographic product graphs. 

The remainder of the paper is structured as follows. Section~\ref{SectionBasicWRDN} covers  basic results on the weak Roman domination number of a   graph. The case of lexicographic product graphs is studied in Sections~\ref{SectionLexicographicProduct}-\ref{Proofs}. Specifically, Section \ref{SectionLexicographicProduct} covers general bounds, Section  \ref{SectionClosedFormulae} covers closed formulae, while   Section \ref{Proofs} contains  the proof of Theorem \ref{pathweightfour}, which is very long. Finally,
in Section \ref{SectionOpenProblems} we collect some open problems  derived from this work.

Throughout the paper, we will use the notation $K_n$, $K_{1,n-1}$, $C_n$, $N_n$ and $P_n$ for complete graphs, star graphs, cycle graphs, empty graphs and path graphs of order $n$, respectively. We use the notation $u \sim v$ if $u$ and $v$ are adjacent vertices,  and $G \cong H$ if $G$ and $H$ are isomorphic graphs. For a vertex $v$ of a graph $G$, $N(v)$ will denote the set of neighbours or \emph{open neighbourhood} of $v$ in $G$, \emph{i.e.} $N(v)=\{u \in V(G):\; u \sim v\}$. The \emph{closed neighbourhood}, denoted by $N[v]$, equals $N(v) \cup \{v\}$.   We denote by $\delta(v)=|N(v)|$ the degree of vertex $v$, as well as $\delta=\min_{v \in V(G)}\{\delta(v)\}$ and $\Delta=\max_{v \in V(G)}\{\delta(v)\}$. The subgraph of $G$ induced by a set $S$ of vertices is denoted  by $\langle S\rangle$.  

 For the remainder of the paper, definitions will be introduced whenever a concept is needed.

\section{Some remarks on $\gamma_r(G)$}\label{SectionBasicWRDN}

In this section we will discuss some basic but useful remarks on the weak Roman domination number of a graph. 
For nonconnected graphs we have the following remark.

\begin{remark}\label{RemarkWeakRomanNonConnectedGraphs}
For any    graph $G$ of $k$ components, $G_1, G_2,\dots ,G_k$, $$\gamma_r(G)=\sum_{i=1}^k \gamma_r(G_k).$$
\end{remark}

According to the remark above,  we can restrict ourselves to the case of  connected graphs.

Recall that a set 
  $D\subseteq V(G)$  is {\it dominating} in $G$ if every vertex in $V(G)\setminus D$ 
has at
least one neighbour in $D$, \textit{i.e}., $N(u)\cap D\ne \emptyset$ for every $u\in V(G)\setminus D$. The {\it domination number} of $G$, denoted by  
$\gamma (G)$, is the minimum cardinality among all dominating sets in $G$. A dominating set of cardinality $\gamma(G)$ is called a $\gamma(G)$-set. The reader is referred to the books \cite{Haynes1998a,Haynes1998}
 for details on domination in graphs.

\begin{remark}{\rm \cite{MR1991720}} \label{DominationChain1}
For any graph $G$,
$$\gamma(G) \le \gamma_r(G)\le \gamma_R(G)\le 2\gamma(G).$$ 
\end{remark}

Graphs with $\gamma_R(G)= 2\gamma(G)$ are called \textit{Roman graphs} \cite{Cockayne2004}. We say that $G$ is a \textit{weak Roman graph} if $\gamma_r(G)= 2\gamma(G)$. 
 Notice that any weak Roman graph is a Roman graph. In general, the converse does not hold. For instance, the graph shown in Figure \ref{FigWeakRoman} is a Roman graph, as $\gamma_R(G)= 2\gamma(G)=4$, while $\gamma_r(G)=3$.

 A \textit{support vertex} of a tree is a vertex adjacent to a leaf, while a  \textit{strong support vertex} is a support vertex that is adjacent to more than one leaf.
 
\begin{lemma}{\rm \cite{MR1991720}}\label{LemmaStrongSupportVertex}
If $T$ is a tree with a unique $\gamma(T)$-set $S$, and if every vertex in $S$ is a  strong support vertex, then $T$ is a weak Roman tree. 
\end{lemma}

The reader is refereed to \cite{MR1991720} for a complete characterization of all weak Roman forests. 

\begin{remark}\label{CaracterizationWRDN=1}
Let $G$ be a graph of order $n$. Then $\gamma_r(G)=1$ if and only if $G\cong K_n$.
\end{remark}

According to this remark, for any noncomplete graph $G$ we have that $\gamma_r(G)\ge 2$. Before discussing the limit case of this trivial bound, we need to introduce some additional notation and terminology. A set $S\subseteq V(G)$ is a \emph{secure dominating set} of $G$  if  $S$ is a dominating set and for every $v\in V(G)\setminus S$ there exists $u\in S\cap N(v)$  such that   $(S\setminus \{u\})\cup \{v\}$ is a dominating set \cite{MR2137919}.  The \emph{secure domination number}, denoted by $\gamma_s(G)$,  is the minimum cardinality among all secure dominating sets. As observed in \cite{MR2137919}, since for any secure dominating set $S$ we can construct a weak Roman dominating function $f(W_0,W_1,W_2)$,  where $W_0=V(G)\setminus S$, $W_1=S$ and $W_2= \emptyset$, we have that $\gamma_s(G)\ge \gamma_r(G).$

\begin{remark}\label{CaractWRDN=2}
Let $G$ be a noncomplete graph. The following statements are equivalent.
\begin{enumerate}[{\rm (i)}]
\item
 $\gamma_r(G)=2$.
\item $\gamma(G)=1$ or $\gamma_s(G)=2$.
\end{enumerate}
\end{remark}

\begin{proof}
 If  $\gamma_r(G)=2$, then 
  any $\gamma_r(G)$-function $f(X_0,X_1,X_2)$ has weight $w(f)=|X_1|+2|X_2|=2$, which implies that either $X_1=\emptyset$ and $|X_2|=1$ or $|X_1|=2$ and $X_2=\emptyset$. In the first case we have $\gamma(G)=1$, and in the second one $X_1$ is a secure dominating set, which implies that  $2=\gamma_r(G)\le \gamma_s(G)\le |X_1|=2$. Therefore, from (i) we deduce  (ii).
  
  Now, since $G$ is not a complete graph, $\gamma_r(G)\ge 2$.  Obviously, if $\gamma(G)=1$, then $\gamma_r(G)= 2$. On the other hand, if there exists a  secure dominating set of cardinality two, then $2\le \gamma_r(G)\le \gamma_s(G)\le 2$. Therefore, from (ii) we deduce (i).
  \end{proof}


Given a graph $G$ and an edge $e\in E(G)$, the graph obtained from $G$ by removing the edge $e$ will be denoted by $G-e$, \textit{i.e}.,  $V(G-e)=V(G)$ and $E(G-e)=E(G)\setminus \{e\}.$ Since any $\gamma_r(G-e)$-function is a WRDF for $G$, we deduce the following basic result.

\begin{remark}{\rm \cite{MR1991720}}\label{RemovingEdges}
For any spanning subgraph $H$ of a graph $G$,
$$\gamma_r(G)\le \gamma_r(H).$$
\end{remark}

\begin{proposition}{\rm \cite{MR1991720}} \label{cycles and pathas}
For any $n\ge 4$,$$\gamma_r(C_n)=\gamma_r(P_n)=\left\lceil\frac{3n}{7}\right\rceil.$$
\end{proposition}

By Remark \ref{RemovingEdges} and Proposition \ref{cycles and pathas} we deduce the following result.

\begin{theorem}
For any Hamiltonian graph $G$ of order $n\ge 4$,
$$\gamma_r(G)\le \left\lceil\frac{3n}{7}\right\rceil.$$
\end{theorem}

Obviously, the bound above is tight, as it is achieved for $G\cong C_n$.

\section{Preliminary results on   lexicographic product   graphs}\label{SectionLexicographicProduct}

Let $G$ and $H$ be two graphs.  The \emph{lexicographic product} of $G$ and $H$ is the graph $G \circ H$ whose vertex set is  $V(G \circ H)=  V(G)  \times V(H )$ and $(u,v)(x,y) \in E(G \circ H)$ if and only if $ux \in E(G)$ or $u=x$ and $vy \in E(H)$.     
Notice that  for any $u\in V(G)$  the subgraph of $G\circ H$ induced by $\{u\}\times V(H)$ is isomorphic to $H$. For simplicity, we will denote this subgraph by $H_u$, and if a vertex of $G$ is denoted by $u_i$, then  the referred  subgraph will be denoted by $H_i$.
For any $u \in V(G)$ and any WRDF $f$ on $G\circ H$ we define $$f(H_u)=\sum_{v\in V(H)}f(u,v) \text{ and }f[H_u]=\sum_{x\in N[u]}f(H_x).$$

\begin{remark} \label{RemarkConnectedLexic}
Let $G$ and $H$ be two graphs. The following assertions hold.

\begin{itemize}
\item $G\circ H$ is connected if and only if   $G$ is connected.
\item  If $G=G_1\cup \ldots \cup G_t,$ then $G\circ H= (G_1\circ H) \cup \ldots \cup (G_t \circ H).$ 
\end{itemize}
\end{remark}

From Remarks \ref{RemarkWeakRomanNonConnectedGraphs} and \ref{RemarkConnectedLexic} we deduce the following result.

\begin{remark}
For any  graph $G$ of components  $G_1, G_2,\dots ,G_k$ and  any graph $H$, $$\gamma_r(G\circ H)=\sum_{i=1}^k \gamma_r(G_i\circ H).$$
\end{remark}


For basic properties of the lexicographic product of two graphs we suggest the  books  \cite{Hammack2011,Imrich2000}.
A main problem in the study of product of graphs consists of finding exact values or sharp
bounds for specific parameters of the product of two graphs and express
these in terms of invariants of the factor graphs. In particular,   we cite the following works on domination theory of lexicographic product graphs. For instance, the domination number was studied in \cite{MR3363260,Nowakowski1996}, the Roman domination number was studied in \cite{SUmenjak:2012:RDL:2263360.2264103}, the rainbow domination number was studied in \cite{MR3057019}, the super domination number was studied in \cite{Dettlaff-LemanskaRodrZuazua2017}, while the doubly connected domination number was studied in \cite{MR3200151}.

To begin our analysis we would point out the following result, which is a direct consequence of Remark    \ref{RemovingEdges}.

\begin{remark}
Let $G$ be a connected graph of order $n$ and let $H$ be a nonempty graph. For any spanning subgraph $G_1$  of $G$, 
 $$\gamma_r(K_n\circ H) \le \gamma_r(G\circ H)\le \gamma_r(G_1\circ H).$$
 In particular, if $G$ is a Hamiltonian graph, then $$\gamma_r(G\circ H)\le \gamma_r(C_n\circ H).$$
\end{remark} 

\subsection{Upper bounds on $\gamma_r(G\circ H)$}

A \emph{total dominating set}  of a graph $G$ with no isolated vertex is a set $S$  of vertices of $G$ such that every vertex in $V(G)$ is adjacent to at least one vertex in $S$.    The\emph{ total domination number} of $G$, denoted by $\gamma_t(G)$, is the cardinality of a smallest total dominating set, and we refer to such a set as a $\gamma_t(G)$-set. Notice that for any graph $G$ with no isolated vertex, 
\begin{equation}\label{DominationChain2}
\gamma_r(G)\le 2\gamma(G)\le 2\gamma_t(G).
\end{equation} 
The reader is referred to the book \cite{Henning2013} for details on total domination in graphs.
This book provides and explores the fundamentals of total domination in graphs.

Notice that if $D$ is a total dominating set  of $G$ and  $h\in V(H)$, then $D\times \{h\}$ is a dominating set of $G\circ H$, so that $\gamma(G\circ H)\le \gamma_t(G)$. Hence, from the first inequality in chain  \eqref{DominationChain2} we deduce the following theorem, which can also be derived from the inequality $\gamma_R(G\circ H)\le 2\gamma_t(G)$ observed in \cite{SUmenjak:2012:RDL:2263360.2264103} and the second inequality in Remark \ref{DominationChain1}.

\begin{theorem}\label{BoundTotalDomination}
If $G$ is a graph with no isolated vertex, then for any  graph   $H$, $$\gamma_r(G\circ H)\le 2\gamma_t(G).$$
\end{theorem}


The total domination number of a path is known and is easy to compute. For every integer $n\ge 3$ we have $\gamma_t(P_n)=\lfloor n/2 \rfloor+\lceil n/4\rceil-\lfloor n/4\rfloor$. We will show in Corollary \ref{CorollaryPaths} that if $\gamma(H)\ge 4$, then $\gamma_r(P_n\circ H)=2\gamma_t(P_n)$.
Thus, the bound above is tight.  Furthermore,
as we will show in Proposition \ref{PropStarsTimesH},  if $n\ge 3$ and $\gamma(H)\ge 4$, then $\gamma_r(K_{1,n}\circ H)=4=2\gamma_t(K_{1,n})$. Notice that $K_{1,n}$ is a graph of diameter two and minimum degree $\delta=1$. In general, for any graph $G$ of diameter two and minimum degree $\delta$, the total domination number of $G$ is bounded above by $\delta +1$. Moreover, if $G$ is a  graph  of  order $n$
with  no  isolated  vertex and maximum degree $\Delta\ge n-2$, then $\gamma_t(G)=2.$ Therefore, the following result is a direct consequence of Theorem \ref{BoundTotalDomination}.

\begin{corollary}\label{BoundDiameterTwo}
The following assertions hold for any graph $H$.
\begin{itemize}
\item If $G$ is a  graph  of  order $n$
with  no  isolated  vertex and maximum degree $\Delta\ge n-2$, then $\gamma_r(G\circ H)\le 4.$
\item If $G$ has diameter two and minimum degree $\delta$, then
$\gamma_r(G\circ H)\le 2(\delta +1).$
\end{itemize}
\end{corollary}

It was shown in \cite{Cockayne1980} that for any connected graph of order $n\ge 3$, $\gamma_t(G)\le \frac{2}{3}n$. Hence,
Theorem \ref{BoundTotalDomination} leads to the following result.

\begin{corollary}\label{UpperBound2n/3}
For any connected graph  $G$   of order $n\ge 3$ and any graph $H$,
$$\gamma_r(G\circ H)\le 2\left\lfloor\frac{2n}{3}\right\rfloor.$$
\end{corollary}

In Proposition \ref{PropositionPeines} we will show that the bound above is tight.


Chellali and Haynes \cite{Chellali2004} established
that the total domination number of a tree $T$ of order $n\ge 3$ is bounded above by $(n+s)/2$, where $s$ is the number of support vertices of  $T$. Therefore, Theorem \ref{BoundTotalDomination} leads to the following corollary.

\begin{corollary}
For any graph $H$ and any tree  $T$    of order $n\ge 3$ having $s$ support vertices,
$$\gamma_r(T\circ H)\le n+s.$$
\end{corollary}

The bound above is tight. For instance, Proposition \ref{PropositionPeines} shows that for any $n=3k$ and any graph $H$ with $\gamma(H)\ge 4$,  $\gamma_r(T_n\circ H)=n+s=4k$, where $T_n$ is a comb defined prior   to Proposition \ref{PropositionPeines}.

As stated by Goddard and Henning \cite{Goddard2002}, if $G$ is a planar graph with diameter two, then $\gamma_t(G)\le 3$. Hence,
as an immediate consequence of Theorem \ref{BoundTotalDomination}, we have the following result.

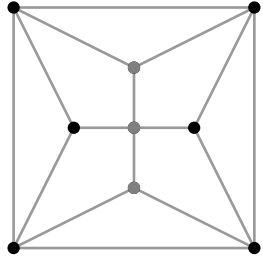
\begin{figure}[h]
\begin{center}
\begin{tikzpicture}
[line width=1pt, scale=0.8]

\coordinate (V1) at (3,0);
\coordinate (V2) at (3,-1);
\coordinate (V3) at (2,0);
\coordinate (V4) at (3,1);
\coordinate (V5) at (4,0);
\coordinate (V6) at (5,-2);
\coordinate (V7) at (1,-2);
\coordinate (V8) at (1,2);
\coordinate (V9) at (5,2);

\draw[black!40]  (V1)--(V2);
\draw[black!40]  (V1)--(V3);
\draw[black!40]  (V1)--(V4);
\draw[black!40]  (V1)--(V5);
\draw[black!40]  (V2)--(V6);
\draw[black!40]  (V2)--(V7);
\draw[black!40]  (V3)--(V7);
\draw[black!40]  (V3)--(V8);
\draw[black!40]  (V4)--(V8);
\draw[black!40]  (V4)--(V9);
\draw[black!40]  (V5)--(V6);
\draw[black!40]  (V5)--(V9);
\draw[black!40]  (V6)--(V7);
\draw[black!40]  (V7)--(V8);
\draw[black!40]  (V8)--(V9);
\draw[black!40]  (V9)--(V6);

\foreach \number in {1,...,9}{
\filldraw[black]  (V\number) circle (0.08cm);
}

\foreach \number in {1,2,4}{
\filldraw[gray]  (V\number) circle (0.08cm);
}

\end{tikzpicture}
\end{center}
\vspace{-0,5cm}
\caption{A planar graph of diameter two.}
\label{fig:g9} 
\end{figure}

\begin{corollary}\label{CorollaryPlanarDiameterTwo}
If $G$ is a planar graph of diameter two, then for any graph $H$,
$$\gamma_r(G\circ H)\le 6.$$
\end{corollary}

The bound above is achieved, for instance, for the planar graph $G$ shown in Figure \ref{fig:g9} and any graph $H$ with $\gamma(H)\ge 4$. An optimum placement of legions in $G\circ H$ can be done by assigning two legions to the copies of $H$ corresponding to the gray-coloured vertices of $G$. 

\begin{corollary}\label{UpperBound2TimesDominationNumber}
For any graph $G$ with no isolated vertex and any  noncomplete graph $H$,
$$\gamma_r(G\circ H)\le  4 \gamma (G).$$
\end{corollary}
\begin{proof}
It is well known that for every  graph $G$ with no isolated vertex,$\gamma_t(G)\le 2\gamma(G)$ (see, for instance, \cite{Bollobas1979(c)}). Hence, by Theorem \ref{BoundTotalDomination} we have $\gamma_r(G\circ H)\le 4\gamma(G)$.
Therefore, the result follows.
\end{proof}

The bound $\gamma_r(G\circ H)\le 4 \gamma (G)$ is achieved, for instance, for graphs $G$ and $H$ satisfying the assumptions of Theorem \ref{ThGeneralizeThe Star}.

\begin{theorem}\label{WeakRomanTimesDominationNumber}   
For any graph $G$  and any  noncomplete graph $H$,
$$\gamma_r(G\circ H)\le   \gamma (G)\gamma_r(H).$$
\end{theorem}
\begin{proof}
Let $f_1(V_0,V_1,V_2)$ be a $\gamma_r(H)$-function and $X$ a $\gamma(G)$-set. Notice that $\gamma_r(H)\ge 2$, as $H$ is not complete.
It is readily seen that $f(W_0,W_1,W_2)$ defined by  $W_1=X\times V_1$ and  $W_2=X\times V_2$ is a WRDF of $G\circ H$. Hence, 
$$\gamma_r(G\circ H)\le |X\times V_1|+2|X\times V_2|=|X|(|V_1|+2|V_2|)=\gamma (G)\gamma_r(H).$$
Therefore, the result follows.
\end{proof}

The bound  $\gamma_r(G\circ H)\le   \gamma (G)\gamma_r(H) $ is achieved, for instance, for any comb graph $T_{3k}$  defined prior  to Proposition \ref{PropositionPeines} and any graph $H$ with $\gamma_r(H) =4$. Besides, the bound is attained for any $G$ and $H$ satisfying the assumptions of Theorem \ref{EqualityWRDNandTwotimesDom}.

A \emph{double total dominating set}  of a graph $G$ with minimum degree greater than or equal to two is a set $S$  of vertices of $G$ such that every vertex in $V(G)$ is adjacent to at least two vertices in $S$, \cite{Henning2013}.    The \emph{double total domination number} of $G$, denoted by $\gamma_{2,t}(G)$, is the cardinality of a smallest double total dominating set, and we refer to such a set as a $\gamma_{2,t}(G)$-set.

\begin{theorem}\label{BoundDoubleDomination}
Let $G$ be a graph of minimum degree greater than or equal to two. The following assertions hold. 

\begin{enumerate}[{\rm (i)}]
\item $\gamma_r(G)\le \gamma_{2,t}(G).$

\item For any graph $H$, $\gamma_{2,t}(G\circ H)\le \gamma_{2,t}(G).$

\item For any graph $H$, $\gamma_{r}(G\circ H)\le \gamma_{2,t}(G).$
\end{enumerate}
\end{theorem}

\begin{proof}
For every $\gamma_{2,t}(G)$-set $S$ we can define a WRDF $f(X_0,X_1,X_2)$ on $G$  by $X_0=V(G)\setminus S$, $X_1=S$ and $X_2=\emptyset$. Hence, (i) follows.

Now, let $D$ be a $\gamma_{2,t}(G)$-set  and $y\in V(H)$. Thus, for every $(x,y)\in V(G)\times V(H)$, there exist $a,b\in D\cap N(x)$, which implies that $(a,y),(b,y)\in (D\times \{y\})\cap N((x,y))$, and so $D\times \{y\}$ is a double total  dominating set of $G\circ H$. Hence, (ii) follows. 

Finally, from (i) and (ii) we deduce (iii), as
$\gamma_{r}(G\circ H) \le \gamma_{2,t}(G\circ H)\le  \gamma_{2,t}(G).$
\end{proof}



In order to show an example of graphs where $\gamma_{r}(G\circ H)= \gamma_{2,t}(G) $, we define the family $\mathcal{G}$    as follows.  A graph $G_{r,s}=(V,E)$ belongs to $\mathcal{G}$ if and only if  there exit two positive integers $r,s$ such that
$V=\{x_1,x_2,x_3,y_1,y_2,\dots, y_r,z_1,$ $z_2,\dots, z_s\}$ and $E=\{x_1y_i:\; 1\le i\le r\}\cup \{x_1z_i:\; 1\le i\le s\}\cup \{x_2y_i:\; 1\le i\le r\}\cup \{x_3z_i:\; 1\le i\le s\}\cup\{x_2x_3\}$. Figure \ref{The graph $G_{4,4}$} shows the graph $G_{4,4}$.

\begin{figure}[h]
\begin{center}
\begin{tikzpicture}
[line width=1pt,scale=0.8]

\foreach \number in {1,...,4}{
\coordinate (A\number) at (\number,0);}
\foreach \number in {6,...,9}{
\coordinate (B\number) at (\number,0);}

\foreach \number in {1,...,4}{
\draw[black!40] (4,-1)-- (A\number) -- (5,1);
}
\foreach \number in {6,...,9}{
\draw[black!40] (6,-1)-- (B\number) -- (5,1);
}
\draw[black!40]  (4,-1) -- (6,-1);

\foreach \number in {1,...,4}{
\filldraw[black]  (A\number) circle (0.08cm)  ;
}
\foreach \number in {6,...,9}{
\filldraw[black]  (B\number) circle (0.08cm)  ;
}
\filldraw[gray]  (5,1) circle (0.08cm)  ;
\filldraw[gray]  (4,0) circle (0.08cm)  ;
\filldraw[gray]  (6,0) circle (0.08cm)  ;
\filldraw[gray]  (4,-1) circle (0.08cm)  ;
\filldraw[gray]  (6,-1) circle (0.08cm)  ;

\node [above] at (5,1) {$x_1$};
\node [below] at (6,-1) {$x_3$};
\node [below] at (4,-1) {$x_2$};
\end{tikzpicture}
\end{center}
\vspace{-0,5cm}
\caption{The set of gray-coloured vertices is a double total dominating set of $G_{4,4}$.}\label{The graph $G_{4,4}$}
\end{figure}
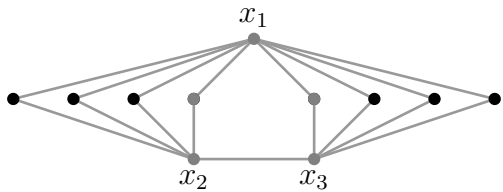

It is not difficult to check  that for any graph $G_{r,s}\in \mathcal{G}$  and any graph $H$ with $\gamma(H)\ge 3$ we have  $\gamma_r(G_{r,s}\circ H)= 5=\gamma_{2,t}(G_{r,s})$.

\begin{corollary}\label{CorollaryBound-n}
For any graph $H$ and any   graph  $G$    of order $n$ and minimum degree greater than or equal to two,  $$\gamma_r(G\circ H)\le n.$$
\end{corollary}

As we will show in Corollary \ref{CorollaryCycles}, the bound above is tight. 

\subsection{Lower bounds on $\gamma_r(G\circ H)$}

In order to deduce our next result we need to state the following basic lemma.

\begin{lemma}\label{Image-one-copy-H}
Let  $G$ be a graph and $H$ a  noncomplete graph. For any  $u\in V(G)$ and any  
$\gamma_r(G\circ H)$-function $f$, 
$$
f[H_u]=\sum_{x\in N[u]}f(H_x)\ge 2.$$
\end{lemma}
\begin{proof}
Suppose that  $f$ is  a $\gamma_r(G\circ H)$-function  and there exists $u\in V(G)$ such that  $f[H_u]\le 1$. If $f[H_u]=1$, then the placement of a legion in a non-universal vertex of $H_u$ produces undefended vertices, which is a contradiction. Now, if $f[H_u]=0$, then there are undefended vertices in $H_u$, which is a contradiction again. Therefore, the result follows. 
\end{proof}

A set $X\subseteq V(G)$ is  called a $2$-\textit{packing}
if $N[u]\cap N[v]=\emptyset $ for every pair of different vertices $u,v\in X$.
 The
$2$-\textit{packing number} $\rho(G)$ is the   cardinality
of any   largest $2$-packing of
$G$. A $2$-packing of cardinality $\rho(G)$ is called a $\rho(G)$-set.

\begin{theorem}\label{MainLowerBound}
For any   graph  $G$ of minimum degree $\delta\ge 1$ and any noncomplete graph $H$, $$\gamma_r(G\circ H)\ge \max \{\gamma_r(G), \gamma_t(G),2\rho(G)\}.$$
\end{theorem}

\begin{proof}
Let $f(W_0,W_1,W_2)$ be a $\gamma_r(G\circ H)$-function. In order to show that $\gamma_r(G\circ H)\ge \gamma_r(G)$, we will show that there exists a WRDF $f_1(X_0,X_1,X_2)$ of $G$ where $X_0=\{x:\; (x,y)\in W_0\}$, $X_2=\{x: \; (x,y)\in W_2\, \text{ or } |\{x\}\times W_1|\ge 2\}$ and $X_1=V(G)\setminus (X_0\cup X_2)$.
Notice that, since $W_1\cup W_2$ is a dominating set of $G\circ H$, $X_1\cup X_2$ is a dominating set of $G$. Now, for every
$(x,y)\in W_0$ there exists $(x',y')\in N((x,y))\cap \left(W_1\cup W_2\right)$  and a  function  $f': V(G\circ H)\longrightarrow \{0,1,2\}$ defined by $f'(x',y')=f(x',y')-1$, $f'(x,y)=1$ and $f'(a,b)=f(a,b)$ for every $(a,b)\not\in \{(x,y),(x',y')\}$, which has no undefended vertex. Hence, the function
$f_1': V(G)\longrightarrow \{0,1,2\}$ defined by
  $f_1'(x')=f_1(x')-1$, $f_1'(x)=1$ and $f_1'(a)=f_1(a)$ for every $a\not\in \{x,x'\}$ has no undefended vertex. Thus, $\gamma_r(G\circ H)\ge \gamma_r(G)$. 
  
  Now, let  $X\subset V(G)$ be a $\rho(G)$-set. By Lemma \ref{Image-one-copy-H} we have
$$\gamma_r(G\circ H)=w(f)= \sum_{u\in V(G)}f(H_u)\ge \sum_{u\in X}f[H_u]\ge 2|X|=2\rho(G),$$
as required. 
  
  In order to prove that $\gamma_r(G\circ H)\ge \gamma_t(G)$, we define  $U_i=\{x\in V(G): \, f(H_x)=i\}$, where $i\in \{0,1\}$, and $U_2=\{x\in V(G): \, f(H_x)\ge 2\}$. By Lemma \ref{Image-one-copy-H} we have that if $x\in U_1$, then there exists $x'\in N(x)\cap (U_1\cup U_2)$. Now, let 
 $U_2^*=\{x\in U_2:\, \sum_{x'\in N(x)}f(H_{x'})=0\}.$ Since $\delta \ge 1$, there exists $U_0^*\subseteq U_0$ such that $|U_0^*|\le |U_2^*|$ with the property  that for every $x\in U_2^*$ there exists $x^*\in U_0^*\cap N(x)$. Notice that $U_1\cup U_2\cup U_0^*$ is a total dominating set.
 Therefore, $$\gamma_{r}(G\circ H)=w(f)=\sum_{u\in V(G)}f(H_u)\ge |U_1\cup U_2\cup U_0^*|\ge \gamma_t(G),$$ as required.
\end{proof}

In the next section we give closed formulae for $\gamma_r(G\circ H)$. In particular, we discuss several cases in which $\gamma_r(G\circ H)=\max \{\gamma_r(G), \gamma_t(G),2\rho(G)\}.$

An example of a graph with $\gamma_r(G)>\max \{ \gamma_t(G),2\rho(G)\}$ is the complete bipartite graph $G\cong K_{3,3}$, where $\gamma_r(K_{3,3})=3$ and  $\gamma_t(K_{3,3})=2\rho(K_{3,3})=2$. An example of a graph with $2\rho(G)>\max \{\gamma_r(G), \gamma_t(G)\}$ is the path graph $P_n$, $n\ge 4$, as   $\gamma_r(P_n)=\left\lceil\frac{3n}{7}\right\rceil$, $\gamma_t(P_n)=\lfloor n/2 \rfloor+\lceil n/4\rceil-\lfloor n/4\rfloor$ and $2\rho(P_n)=2\gamma(P_n)=2\left\lceil\frac{n}{3}\right\rceil$.  Finally, for the graph shown in Figure \ref{We compare total, weak Roman and 2-packing} we have $\gamma_t(G)=5>4=\max \{\gamma_r(G), 2\rho(G)\}$. 

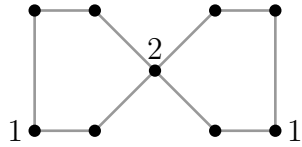
\begin{figure}[h]
\begin{center}
\begin{tikzpicture}
[line width=1pt,scale=0.8]
\draw[black!40]  (0,0) -- (1,1)--(2,1)--(2,-1)--(1,-1)--(0,0)--(-1,1)--(-2,1)--(-2,-1)--(-1,-1)--(0,0);

\filldraw[black]  (0,0) circle (0.08cm)  ;
\filldraw[black]  (1,1) circle (0.08cm)  ;
\filldraw[black]  (2,1) circle (0.08cm)  ;
\filldraw[black]  (2,-1) circle (0.08cm)  ;
\filldraw[black]  (1,-1) circle (0.08cm)  ;
\filldraw[black]  (-1,1) circle (0.08cm)  ;
\filldraw[black]  (-2,1) circle (0.08cm)  ;
\filldraw[black]  (-2,-1) circle (0.08cm)  ;
\filldraw[black]  (-1,-1) circle (0.08cm)  ;

\node [above] at (0,0) {$2$};
\node [right] at (2,-1) {$1$};
\node [left] at (-2,-1) {$1$};
\end{tikzpicture}
\end{center}\vspace{-0,5cm}
\caption{$\gamma_r(G)=4$, the labels correspond to an optimum placement of legions. }\label{We compare total, weak Roman and 2-packing}
\end{figure}

It is well known that for any graph $G$, $\gamma(G)\ge \rho(G)$.
Meir and Moon \cite{MR0401519} showed in 1975 that $\gamma(T)= \rho(T)$ for any tree $T$.  We remark that in general, these $\gamma(T)$-sets and $\rho(T)$-sets  are not identical. Notice that for any weak Roman tree $T$ we have $\gamma_r(T)=2\rho(T)$, while if $T$ is not a weak Roman tree, then $\gamma_r(T)<2\gamma(G)=2\rho(T)$.

\begin{corollary}
For any tree $T$ and any noncomplete graph $H$,
$$\gamma_r(T\circ H)\ge 2\gamma(T). $$
\end{corollary}

The bound above is achieved for any tree $T$ and any graph $H$ satisfying the assumptions of Theorem \ref{Tree2GammaT}.


\section{Closed formulae for $\gamma_r(G\circ H)$ }
\label{SectionClosedFormulae}

To begin this section we consider the case of lexicographic product graphs in which the second factor is a complete graph.

\begin{proposition}For any  graph $G$ and any  integer  $n\ge 1$,
  $$\gamma_{r}(G\circ K_{n})=  \gamma_r(G).$$
\end{proposition}

\begin{proof}
The result is straightforward. 
We leave the details to the reader.
\end{proof}


 From Theorems \ref{BoundTotalDomination} and \ref{MainLowerBound} we have the following result.

\begin{theorem}\label{Equalityto2timestotaldomination}
For any graph  $G$  with
  $\gamma_t(G)=\frac{1}{2}\max \{\gamma_r(G),2\rho(G)\}$ and any noncomplete graph $H,$  $$\gamma_r(G\circ H)=2\gamma_t(G).$$
\end{theorem}

To show some families of graphs for which  $\gamma_r(G)=2\gamma_t(G)=2\rho(G)$, we introduce the corona product of two graphs. Let $G_1$ be a graph  of order $n$ and let $G_2$ be a graph. The \emph{corona product} of $G_1$ and $G_2$, denoted by  $G_1\odot G_2$, was defined in  \cite{Frucht1970} as the graph obtained from $G_1$ and $G_2$ by taking one copy of $G_1$ and $n$ copies of $G_2$ and joining by an edge each vertex from the $i$-th copy of $G_2$ with the $i$-th vertex of $G_1$.

\begin{theorem}
For any  graph $G_1$ with no isolated vertex and  any noncomplete graph   $G_2$,  $$\gamma_r(G_1\odot G_2)=2\gamma_t(G_1\odot G_2)=2\rho(G_1\odot G_2).$$
\end{theorem}

\begin{proof}
Since $\gamma(G_1\odot G_2)=|V(G_1)|$, we have that $\gamma_r(G_1\odot G_2)\le 2|V(G_1)|$. Now, we denote by $\langle g_i\rangle+G_2$ the subgraph of $G_1\odot G_2$ induced by $g_i\in V(G)$ and the vertex set of the $i$-th copy of $G_2$. Since $G_2$ has two nonadjacent vertices and $g_i$ is  the only vertex of $\langle g_i\rangle+G_2$  which is adjacent to some vertex outside $\langle g_i\rangle+G_2$, we deduce that every $\gamma_r(G_1\odot G_2)$-function assigns at least two legions to the vertex set of $\langle g_i\rangle+G_2$, which implies that $\gamma_r(G_1\odot G_2)\ge  2|V(G_1)|$. Now, since $G_1$ is a graph with no isolated vertex, $V(G_1)$ is a total dominating set. Hence, $\gamma_r(G_1\odot G_2)=2|V(G_1)|=2\gamma_t(G_1\odot G_2).$ 

The proof of the equality $\gamma_t(G_1\odot G_2)=\rho(G_1\odot G_2)$ is straightforward.
\end{proof}


If $\gamma_r(G)=2\gamma(G)$, then for the Cocktail-party graph $K_{2k}-F$, where $F$ is a perfect matching of $K_{2k}$, we have $\gamma_r(G\circ (K_{2k}-F))=\gamma_r(G)$. This example  is a particular case  of the next result which is derived from Theorems \ref{WeakRomanTimesDominationNumber} and  \ref{MainLowerBound}.

\begin{theorem}\label{EqualityWRDNandTwotimesDom}
For any weak Roman graph  $G$  and any graph $H$ such that $\gamma_r(H)=2$, $$\gamma_r(G\circ H)=2\gamma(G).$$
\end{theorem}

The study of weak Roman graphs  was initiated in 
\cite{MR1991720} by Henning and Hedetniemi, where they characterized forests for which the equality holds.   The general problem 
 of characterizing all weak Roman  graphs  remains open.

From Lemma \ref{LemmaStrongSupportVertex} 
 and  Theorem \ref{EqualityWRDNandTwotimesDom} we derive the following result. 

\begin{theorem}\label{Tree2GammaT}
If $T$ is a tree with a unique $\gamma(T)$-set $S$, and if every vertex in $S$ is a  strong support vertex, then for any graph $H$ with $\gamma_r(H)=2$, $$\gamma_r(T\circ H)=2\gamma (T).$$
\end{theorem}

Our next result shows that the inequality $\gamma_r(G\circ H)\le  4\gamma(G) $ stated in Corollary \ref{UpperBound2TimesDominationNumber} is tight. 

\begin{theorem}\label{ThGeneralizeThe Star}
If $G$ is a graph with $\gamma_t(G)=2\gamma(G)$  and there exists a $\gamma(G)$-set  $D$ such that every vertex in $D$ is adjacent to a vertex of degree one, then for any graph $H$ with 
 $\gamma(H)\ge 4$,
$$\gamma_r(G\circ H)= 4\gamma(G).$$
\end{theorem}

\begin{proof}
Assume that $\gamma_t(G)=2\gamma(G)$,  $\gamma(H)\ge 4$ and let $D$  be a $\gamma(G)$-set such that every vertex in $D$ is adjacent to a vertex of degree one. We will show that $\gamma_r(G\circ H)\ge 4\gamma(G)$.
Since  $\gamma_t(G)=2\gamma(G)$,  the vertex  set of $G$ can be partitioned by the  closed neighbourhoods of vertices in $D$, \textit{i.e.},  $V(G)=\cup_{x\in D}N[x]$ and $N[x]\cap N[y]=\emptyset$, for every $x,y\in D$, $x\ne y$.  Now, let $f(W_0,W_1,W_2)$ be a $\gamma_r(G\circ H)$-function and 
let $x'\in N(x)$ be a vertex of degree one, for  $x\in D$. 
Suppose that $f$ assigns at most three legions to $N[x]\times V(H)$.
We differentiate the following cases for the set $W=W_1\cup W_2$.

\noindent Case 1. $|W\cap (\{x\}\times V(H))|=3$. Since $\gamma(H)\ge 4$, there exists at least one vertex in  $\{x\}\times V(H)$ which is not dominated by the elements in $W$, which is a contradiction.

\noindent Case 2. $|W_2\cap (\{x\}\times V(H))|=1$ or $|W_1\cap (\{x\}\times V(H))|=2$. In both cases 
 there  exists $y\in N(x)$ such that  $|W_1\cap (\{y\}\times V(H))|=1$. Since $\gamma(H)\ge 4$, the movement of a legion from the vertex in $W_1\cap (\{y\}\times V(H))$ to any vertex in  $W_0\cap (\{x\}\times V(H))$ produces  undefended vertices in $\{x\}\times V(H)$, which is a contradiction.

\noindent Case 3. $|W_1\cap (\{x\}\times V(H))|=1$. Since  $\gamma(H)\ge 4$, the movement of a legion from the vertex in  $W_1\cap (\{x\}\times V(H))$ to any vertex in $ W_0\cap ( \{x'\}\times V(H))$ produces  undefended vertices in $\{x'\}\times V(H)$, which is a contradiction.

\noindent Case 4. $|W\cap (\{x\}\times V(H))|=0$. Since $\gamma(H)\ge 4$, there exists at least one vertex  $(x',h)\in W_0$ which is not dominated by the elements in $W$, which is a contradiction.

According to the four cases above, for every $x\in D$ we have that $f$ assigns at least four legions to $ N[x]\times V(H)$, which implies that $\gamma_r(G\circ H)\ge 4\gamma(G)$. 

Furthermore, by Corollary \ref{UpperBound2TimesDominationNumber}, $\gamma_r(G\circ H)\le 4\gamma(G)$. Therefore, the result follows.
\end{proof}

\begin{figure}[h]
\begin{center}
\begin{tikzpicture}
[line width=1pt,scale=0.8]

\foreach \number in {1,...,9}{
\coordinate (A\number) at (\number,0);}
\foreach \number in {2,5,8}{
\coordinate (B\number) at (\number,1);
\coordinate (C\number) at (\number,-1);}
\foreach \number in {2,5,8}{
\draw[black!40] (C\number)--(A\number)--(B\number);
}

\draw[black!40]  (A1)--(A2)--(A3)--(A4)--(A5)--(A6)--(A7)--(A8)--(A9);

\foreach \number in {1,...,9}{
\filldraw[black]  (A\number) circle (0.08cm)  ;
}
\foreach \number in {2,5,8}{
\filldraw[black]  (B\number) circle (0.08cm)  ;
\filldraw[black]  (C\number) circle (0.08cm)  ;
}
\filldraw[gray]  (2,0) circle (0.08cm)  ;
\filldraw[gray]  (5,0) circle (0.08cm)  ;
\filldraw[gray]  (8,0) circle (0.08cm)  ;

\end{tikzpicture}
\end{center}
\vspace{-0,5cm}
\caption{The set of gray-coloured vertices is the only dominating set of $G$. In this case $\gamma_t(G)=2\gamma(G)=6$. By Theorem \ref{ThGeneralizeThe Star}, for any graph $H$ with $\gamma(H)\ge 4$ we have $\gamma_r(G\circ H)=12=4\gamma(G).$}\label{Only dominating set}
\end{figure}
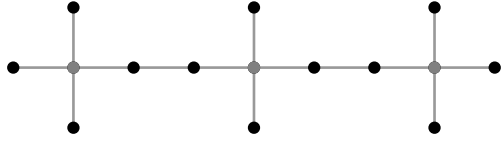

  For the tree shown in Figure \ref{Only dominating set} we have $\gamma(T)= \rho(T)=3$. Notice that in this case the set of support vertices of $T$ is the only $\gamma(T)$-set and a $\rho(T)$-set.

\begin{corollary}
 If the set of support vertices of a tree $T$ is a $\rho(T)$-set, then for any graph $H$ with $\gamma(H)\ge 4$,
$$\gamma_{r}(T\circ H)= 4\gamma(T).$$
\end{corollary}

From Theorems \ref{BoundDoubleDomination} and \ref{MainLowerBound} we have the following result.

\begin{theorem}\label{WRoman=2total}
If $G$ is a graph such that $\gamma_{2,t}(G)=\max \{\gamma_r(G),2\rho(G)\}$, then for any noncomplete graph $H$, $$\gamma_r(G\circ H)=\gamma_{2,t}(G).$$
\end{theorem}

According to Theorem \ref{WRoman=2total}, the problem of characterizing  the graphs for which $\gamma_{2,t}(G)=\gamma_r(G)$  or $\gamma_{2,t}(G)=2\rho(G)$ deserves being considered in future works.
We will construct a family $\mathcal{H}_k$ of graphs such that $\gamma_r(G)=\gamma_{2,t}(G)$, for every $G\in \mathcal{H}_k$.
 A graph $G=(V,E)$ belongs to $\mathcal{H}_k$ if and only if  it is constructed from a cycle $C_k$ and $k$ empty graphs $N_{s_1},\dots, N_{s_k}$ of order $s_1,\dots,s_k$, respectively, and joining by an edge each vertex from $N_{s_i}$ with the  vertices $v_i$ and $v_{i+1}$ of $C_k$. Here we are assuming that $v_i$ is adjacent to $v_{i+1}$ in $C_k$, where the subscripts are taken modulo $k$.
 Figure \ref{The graphs in H}  shows a graph belonging to $\mathcal{H}_k$, where $k=4$, $s_1=s_3=3$ and  $s_2=s_4=2$.
\begin{figure}[h]
\begin{center}
\begin{tikzpicture}
[line width=1pt,scale=0.8]

\foreach \number in {-4,-3,-2}{
\coordinate (A\number) at (\number,0);}
\foreach \number in {-3,-2}{
\coordinate (C\number) at (0,\number);}
\foreach \number in {2,3}{
\coordinate (D\number) at (0,\number);}
\foreach \number in {2,3,4}{
\coordinate (B\number) at (\number,0);}

\foreach \number in {-4,-3,-2}{
\draw[black!40] (-1,-1)-- (A\number) -- (-1,1);
}
\foreach \number in {2,3}{
\draw[black!40] (-1,1)-- (D\number) -- (1,1);}
\foreach \number in {2,3,4}{
\draw[black!40] (1,-1)-- (B\number) -- (1,1);
}
\foreach \number in {-2,-3}{
\draw[black!40] (-1,-1)-- (C\number) -- (1,-1);}
\draw[black!40]  (-1,-1) -- (1,-1)--(1,1)--(-1,1)--(-1,-1);

\filldraw[gray]  (1,1) circle (0.08cm)  ;
\filldraw[gray]  (-1,1) circle (0.08cm)  ;
\filldraw[gray]  (1,-1) circle (0.08cm)  ;
\filldraw[gray]  (-1,-1) circle (0.08cm)  ;

\foreach \number in {2,3,4}{
\filldraw[black]  (B\number) circle (0.08cm)  ;
}
\foreach \number in {-2,-3,-4}{
\filldraw[black]  (A\number) circle (0.08cm)  ;
}
\foreach \number in {2,3}{
\filldraw[black]  (D\number) circle (0.08cm)  ;
}
\foreach \number in {-2,-3}{
\filldraw[black]  (C\number) circle (0.08cm)  ;
}
\end{tikzpicture}
\end{center}
\vspace{-0,5cm}
\caption{The set of gray-coloured vertices is a double dominating set.}\label{The graphs in H}
\end{figure}
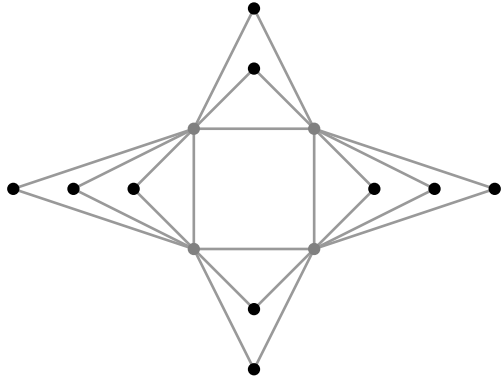

For any graph $G\in \mathcal{H}_k$ we have $\gamma_r(G)=\gamma_{2,t}(G)=k$. Therefore, by Theorem \ref{WRoman=2total}, for any $G\in \mathcal{H}_k$ and any graph $H$,  $\gamma_r(G\circ H)=k.$

From now on we say that a vertex $a\in V(H)$ satisfies Property $\mathcal{P}$ if $\{a,b\}$ is a dominating set of $H$, for every $b\in V(H)\setminus N[a]$. In other words,
$a\in V(H)$ satisfies  $\mathcal{P}$ if the subgraph induced by $V(H)\setminus N[a]$ is a clique.

\begin{proposition}For any  integer $n\ge 3$    and any  noncomplete
graph $H$,
$$2\le \gamma_r( K_n\circ H)\le 3.$$
Furthermore, $\gamma_r( K_n\circ H)=2$ if and only if $\gamma_r(H)=2$ or there exists a vertex of $H$ which satisfies $\mathcal{P}$.
\end{proposition}

\begin{proof}
By Remark \ref{CaracterizationWRDN=1} we have  $\gamma_r(K_n\circ H)\ge 2$ and by Theorem \ref{BoundDoubleDomination} we have that $\gamma_r( K_n\circ H)\le 3.$

To characterize the graphs with $\gamma_r(K_n\circ H)= 2$ we  first assume that  $\gamma_r(H)=2$, and we will apply Remark \ref{CaractWRDN=2} to the graph $H$.
Let $u\in V(G)$ and let $\{a,b\}\subseteq V(H)$ a secure dominating set. We claim that the function $f(X_0,X_1,X_2)$ defined by
 $X_0=(V(K_n)\times V(H))\setminus \{(u,a),(u,b)\}$, $X_1=\{(u,a),(u,b)\}$ and $X_2=\emptyset$ is a $\gamma_r(K_n\circ H)$-function. To see this, we only need to observe that the movement of a legion from $(u,a)$ (or from $(u,b)$) to a vertex in $X_0$ does not produce undefended vertices.   Now, if $\gamma(H)=1$, then we define the $\gamma_r(K_n\circ H)$-function $f$ by $X_0=V(K_n)\times V(H)\setminus \{(u,z)\}$, $X_1=\emptyset$ and $X_2=\{(u,z)\}$, where $z\in V(H)$ is a vertex of maximum degree. On the other hand, if $a\in V(H)$  satisfies  $\mathcal{P}$, then
we define the $\gamma_r(K_n\circ H)$-function $f$ by $X_0=(V(K_n)\times V(H))\setminus \{(u_1,a),(u_2,a)\}$, $X_1=\{(u_1,a),(u_2,a)\}$ and $X_2=\emptyset$, where $u_1\ne u_2$.

Conversely, assume that $\gamma_r( K_n\circ H)=2$ and let $f(W_0,W_1,W_2)$
be a $\gamma_r(K_n\circ H)$-function. Notice  that, $|W_1|+2|W_2|=2$.
 Now, if  $W_2=\{(u,a)\}$, then $\gamma(H)=1$.  From now on, assume that $W_1=\{(u_1,a),(u_2,b)\}$ and $\gamma(H)\ge 2$.
If $u_1=u_2$, then $\{a,b\}$ is a secure dominating set of $H$, and by Remark \ref{CaractWRDN=2}, $\gamma_r (H)=2$.
Finally, if  $u_1\ne u_2$, then the movement of a legion from $(u_2,a)$ to $(u_1,c)$, where $c\in V(H)\setminus N[a]$, does not produce undefended vertices, which implies that $a$ satisfies $\mathcal{P}$.
 \end{proof}

\begin{proposition} \label{PropStarsTimesH}
Let $H$ be a graph and let $n\ge3$ be an integer. Then the following statements hold.
\begin{enumerate}[{\rm (i)}]
\item If $\gamma_r(H)\in \{2,3\}$, then $\gamma_r(K_{1,n}\circ H)=\gamma_r(H).$
\item If $\gamma_r(H)\ge 4$, then $3\le \gamma_r(K_{1,n}\circ H)\le 4.$

\item If $\gamma(H)\ge 4$, then $\gamma_r(K_{1,n}\circ H)=4.$
\end{enumerate}
\end{proposition}

\begin{proof}
Let $ u_0$ be a universal vertex of $K_{1,n}$.  By Remark \ref{CaracterizationWRDN=1} we have that $\gamma_r(K_{1,n} \circ H)\ge 2$ and by  Theorem \ref{BoundTotalDomination}, $\gamma_r(K_{1,n}\circ H)\le 2\gamma_t(K_{1,n})=4.$

Let  $g$ be  a $\gamma_r(H)$-function. Assume that $\gamma_r(H)\in \{2,3\}$. The function
$f: V(K_{1,n}\circ H)\longrightarrow \{0,1,2\}$ defined by $f(u_0,v)=g(v)$, for every $v\in V(H)$, and $f(u,v)=0$, for every $u\in V(K_{1,n})\setminus \{u_0\}$ and  $v\in V(H)$, 
is a WRDF of $K_{1,n}\circ H$, which implies that $\gamma_r(K_{1,n}\circ H)\le\gamma_r(H).$  Hence, if $\gamma_r(H)=2$, then we are done. Since $n\ge 3$,  
for any $\gamma_r(K_{1,n}\circ H)$-function  we have $f(H_{u_0})\ge 2$ and, if $\gamma_r(H)\ge 3$, then  $w(f)\ge 3$. Thus, (i) and (ii) follow. 

%

Finally, if  $\gamma(H)\ge 4$, then Theorem \ref{ThGeneralizeThe Star} leads to $\gamma_r(K_{1,n}\circ H)= 4$.
\end{proof}

We will now show that  the bound given in Corollary \ref{UpperBound2n/3} is tight. To this end, we need to introduce some additional notation.  
Given a graph $G$, let $\mathcal{P}_3(G)$ be the family of ordered sets $S=\{x_1,x_2,x_3  \} \subset V(G)$ such that $\langle S \rangle \cong P_3$,   $\delta (x_1)\geq 2$, $\delta (x_2)=2$ and $\delta (x_3)=1$.


\begin{lemma}
\label{lemma:p3}
Let $G$ and $H$ be two graphs, and $\{x_1,x_2,x_3\}\in  \mathcal{P}_3(G)$. If $\gamma(H) \ge 4$, then for any  $\gamma_r(G\circ H)$-function $f$,
 $$\sum_{i=1}^3f(H_i) = 4.$$
 Furthermore, 
 there exists a $\gamma_r(G\circ H)$-function $f$, such that
 $f(H_2) = 2$ and $f(H_3)  = 0.$
\end{lemma}

\begin{proof}
 Suppose that there exists    a $\gamma_r(G\circ H)$-function $f$  with $$\sum_{i=1}^3f(H_i)  \le 3.$$ 
 We differentiate the following cases according to the  value of $f(H_1)$. 

\begin{enumerate}[1.]

\item $ f(H_1)=0$. If $f(H_2)=0$ (resp. $f(H_3)=0$), then there is an undefended vertex in  $H_3$ (resp. $H_2$). If $f(H_2)=1$ (resp. $f(H_3)=1$), then the movement of the legion from $H_2$ to $H_3$ (resp. from $H_3$ to $H_2$) produces an undefended vertex in  $H_3$   (resp. from $H_2$).

\item  $f(H_1)=1$. If $f(H_2)=0$, then there is an undefended vertex in  $H_3$. If $f(H_2)=1$, then the movement of the legion from $H_2$ to $H_3$ produces an undefended vertex in  $H_3$. Finally, If $f(H_2)=2$, then the movement of the legion from $H_1$ to $H_2$ produces an undefended vertex in  $H_2$. 

\item  $ f(H_1)=2$. If $f(H_2)=0$, then there is an undefended vertex in  $H_3$. If $f(H_2)=1$, then the movement of the legion from $H_2$ to $H_3$ produces an undefended vertex in  $H_3$. 

\item  $f(H_1)=3$. In this case the vertices in $H_3$ are 
 undefended. 
\end{enumerate}

In all cases above we obtain a contradiction, which implies that $f(H_1)+f(H_2)+f(H_3)\ge 4$. 
To conclude the proof we only need to observe that we can construct a $\gamma_r(G\circ H)$-function  $f$ with  $f(H_1)+f(H_2)+f(H_3) =4$, as we can take $f(H_1)=f(H_2)=2$ and  $f(H_3) =0$. 
\end{proof}

We will now prove that there exists a family of trees $T_n$, which we will call \textit{combs}, such that for any graph   $H$   with $\gamma(H)\ge 4$, $\gamma_r(T_n \circ H) = 2\left\lfloor\frac{2n}{3}\right\rfloor.$ With this end we will now describe this family of combs. Take a path $P_k$ of length $k=\lceil \frac{n}{3} \rceil$, with vertices $v_1, \ldots, v_k$, and attach a path $P_3$ to each vertex $v_1, \ldots, v_{k-1}$, by identifying each $v_i$ with a leaf of its corresponding copy of $P_3$. Finally, we attach a path of length $r = n - 3\lceil \frac{n}{3} \rceil +2$ to $v_k$. Notice  that 

\begin{displaymath}
n - 3\left\lceil \frac{n}{3} \right\rceil + 2 = 
\begin{cases}
0 \mbox{ if } n \equiv 1 \, (\mbox{mod } 3); \\
1 \mbox{ if } n \equiv 2 \, (\mbox{mod } 3); \\
2 \mbox{ if } n \equiv 0 \, (\mbox{mod } 3). 
\end{cases}
\end{displaymath}
Figure \ref{fig:tn} shows the construction of $T_n$ for different values of $n$. Notice that the comb of order six  is simply $T_6\cong P_6$.


\begin{figure}[h]
\centering
\begin{tikzpicture}[transform shape, inner sep = .7mm]

\foreach \k in {0,5,10}{
\foreach \i in {1,2,3}{
\foreach \j in {1,2,3,4}{
	\node [draw=black, shape=circle, fill=black]  at  (\i+\k,\j) {};
}
}
\foreach \i in {1,2,3,4}{
	\draw (1+\k,\i)--(2+\k,\i)--(3+\k,\i);
}
\draw (1+\k,4)--(1+\k,3)--(1+\k,2);
\draw (1+\k,1)--(1+\k,0);
\node [] at (1+\k,1.60) {$\vdots$};
}

\node [draw=black, shape=circle, fill=black]  at  (1,0) {};
\node [draw=black, shape=circle, fill=black]  at  (6,0) {};
\node [draw=black, shape=circle, fill=black]  at  (7,0) {};
\node [draw=black, shape=circle, fill=black]  at  (11,0) {};
\node [draw=black, shape=circle, fill=black]  at  (12,0) {};
\node [draw=black, shape=circle, fill=black]  at  (13,0) {};

\draw (6,0)--(7,0);
\draw (11,0)--(12,0)--(13,0);

\end{tikzpicture}

\caption{$T_n$ for $r = 1, 2, 0$.}
\label{fig:tn} 
\end{figure}
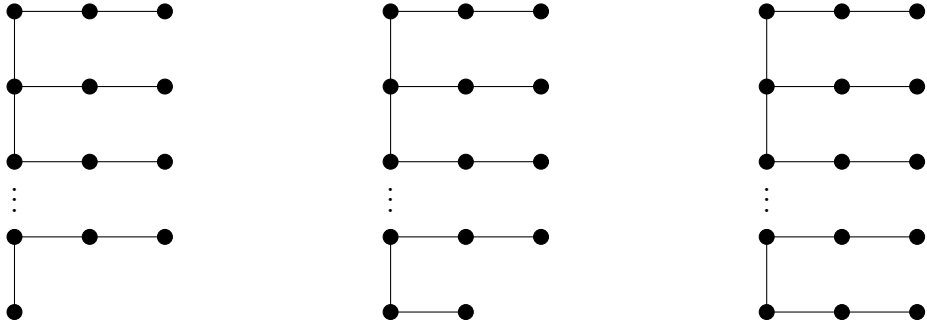

\begin{proposition} \label{PropositionPeines}
For any $n\ge 4$ and any graph   $H$   with $\gamma(H)\ge 4$,
$$\gamma_r(T_n \circ H) = 2\left\lfloor\frac{2n}{3}\right\rfloor.$$ 
\end{proposition} 

\begin{proof}
By Corollary \ref{UpperBound2n/3} we have $\gamma_r(T_n \circ H) \le 2\left\lfloor\frac{2n}{3}\right\rfloor.$ In order to show that $\gamma_r(T_n \circ H) \ge 2\left\lfloor\frac{2n}{3}\right\rfloor$ we differentiate three cases.

 If $n =3k$, then Lemma \ref{lemma:p3} leads to $\gamma_r(T_n \circ H) = 4k= 2\left\lfloor\frac{2n}{3}\right\rfloor$. Now, if  $n=3(k-1)+1$, then Lemma \ref{lemma:p3} leads to $\gamma_r(T_n \circ H) \ge  4(k-1)= 2\left\lfloor\frac{2n}{3}\right\rfloor$. Finally, 
 if  $n=3(k-1)+2$, then Lemma \ref{lemma:p3} leads to $\gamma_r(T_n \circ H) \ge  4(k-1)+2= 2\left\lfloor\frac{2n}{3}\right\rfloor$.
\end{proof}

Given a graph $G$, let $\mathcal{P}_4(G)$ be  the family of ordered sets $S=\{x_1,x_2,x_3,$ $x_4 \} \subset V(G)$ such that $\langle S \rangle \cong P_4$,   $\delta (x_1)\geq 2$, $\delta (x_2)=\delta (x_3)=2$ and $\delta (x_4)\geq 2$.
For any $G$ such that $\mathcal{P}_4(G)\ne \emptyset$ we define the family $\mathcal{O}_4(G)$ of graphs $G^*$ constructed from $G$ as follows. Let $S\in \mathcal{P}_4(G) $ such that $\langle S\rangle =P_4=(x_1,x_2,x_3,x_4)$, $X=\{x_1x_2,x_2x_3,x_3x_4\}$ and $Y=\{ab:\, a\in N(x_1)\setminus \{x_2\} \text{ and } b\in N(x_4)\setminus \{x_3\}\}$. The vertex set of $G^*$ is  $V(G^*)=V(G)\setminus S$
 and the edge set is $E(G^*)=\left(E(G)\setminus X\right) \cup Y$.

\begin{theorem}\label{pathweightfour}
Let $G$ be a graph such that $\mathcal{P}_4(G)\ne \emptyset$ and let $H$ be a graph. If  $\gamma(H)\geq 4$, then for any $G^* \in \mathcal{O}_4(G)$,
$$\gamma_r(G\circ H)=\gamma_r(G^*\circ H)+4.$$
\end{theorem}

\begin{proof}We will prove this result in Section \ref{Proofs}.
\end{proof}

A simple case analysis shows that for $n\in \{3,4,5,6\}$ and  any graph $H$ such that $\gamma( H)\ge 4$ we have $\gamma_r(C_n\circ H)=n$. Hence, Theorem \ref{pathweightfour} immediately leads to the following corollary. 

\begin{corollary}\label{CorollaryCycles}
Let $n\ge 3$ be an integer and let $H$ be a graph. If  $\gamma(H)\geq 4$, then
$$\gamma_r(C_n\circ H)=n.$$
\end{corollary}

It is readily seen that if $\gamma( H)\ge 4$, then 
$\gamma_r(P_2\circ H)=\gamma_r(P_3\circ H)=\gamma_r(P_4\circ H)=4$ and  $\gamma_r(P_5\circ H)=6$ . Therefore, Theorem \ref{pathweightfour}   leads to the following result.

\begin{corollary}\label{CorollaryPaths}
Let $n\ge 2$ be an integer and let $H$ be a graph. If  $\gamma(H)\geq 4$, then
\begin{eqnarray*}\gamma_r(P_n\circ H)= \left \{ \begin{array}{ll}

n, & n\equiv 0\pmod 4 ;
\\
n+2, & n\equiv 2\pmod 4;
\\
n+1, & \text{otherwise}.
\end{array} \right
.\end{eqnarray*}
\end{corollary}

\section{Proof of Theorem \ref{pathweightfour}}\label{Proofs}

To prove Theorem \ref{pathweightfour} we need the following lemma.

\begin{lemma}\label{twoouterweights}
Let $G$ and $H$ be nontrivial connected graphs. If $\gamma(H)\ge 4$, then there exists a $\gamma_r(G\circ H)$-function $f$ such that  $\sum_{u'\in N(u)} f(H_{u'})\geq 2,$
for every $u\in V(G)$.
\end{lemma}

\begin{proof}
Let  $u,u'\in V(G)$ such that $u'\in N(u)$ and $v'\in V(H)$. First, suppose that $\sum_{z\in N(u)}  f(H_{z})=f(u',v')=1$. If $f(H_u)<\gamma(H)-1$, then there exists $v\in V(H)$ such that $\sum_{h\in N[v]}f(u,h)=0$, so that the movement of the legion from $(u',v')$ to  $(u,v)$ produces undefended vertices, which is a contradiction. Hence, $f(H_u)\ge \gamma(H)-1\ge 3$ and we can construct a $\gamma_r(G\circ H)$-function $f_1$ from $f$ as follows. For some $(u,v)$ such that $f(u,v)\ge 1$ we set  $f_1(u,v)=f(u,v)-1$, for some $v''\ne v'$ we set   $f_1(u',v'')=1$ and $f_1(x,y)=f(x,y)$ for every $(x,y)\in V(G\circ H)\setminus \{(u,v),(u',v'')\}$. Hence,  $\sum_{z\in N(u)}  f_1(H_{z})=f_1(u',v')+f_1(u',v'')=2$.

Now, if $\sum_{z\in N(u)}  f(H_{z})=0, $  then we proceed as above to construct a $\gamma_r(G\circ H)$-function $f_1$ from $f$ by the movement of two legions from $H_u$ to $(u',v')$. In this  case,  $\sum_{z\in N(u)}  f_1(H_{z})=f_1(u',v')=2$.

For each $u\in V(G)$ such that $\sum_{u'\in N(u)}  f(H_{u'})\le 1$  we can repeat the procedure above until finally obtaining a $\gamma_r(G\circ H)$-function satisfying the result.
\end{proof}

\begin{proof}[The proof of Theorem \ref{pathweightfour}] 
Let $S\in \mathcal{P}_4(G) $ such that $\langle S\rangle\cong P_4=(x_1,x_2,x_3,x_4)$.  We will first show that $\gamma_r(G\circ H)\leq \gamma_r(G^*\circ H)+4$. Let $f$ be a $\gamma_r(G^*\circ H)$-function and define
  $\alpha_1$ and $\alpha_4$ as follows:
$$\alpha_1=\sum_{\begin{array}{c}
 x\in N(x_1)\setminus \{x_2\}\\
\end{array} } f(H_x)\;\;\text{ and }\;\;\alpha_4=\sum_{\begin{array}{c}
 x\in N(x_4)\setminus\{x_3\}\\
\end{array}} f(H_x).$$
We will construct a WRDF $f_1$ on $G\circ H$ from $f$ such that $w(f_1)\leq w(f)+4$. For each vertex $(u,v)\in V(G^*\circ H)$ we set $f_1(u,v)=f(u,v)$ and now we will describe the following six cases for the vertices   $(u,v)\in S\times V(H)$, where symmetric cases are omitted. In all these cases we fix $y\in V(H)$. 

\begin{enumerate}
\item  $\alpha_1\geq 2\text{ and }\alpha_4\geq 2$.  We  set $f_1(x_1,y)=f_1(x_4,y)=2$  and $f_1(u,v)=0$ for every $(u,v)\notin \{(x_1,y),(x_4,y)\}$. 

\item   $\alpha_1\geq 2\text{ and }\alpha_4=1$.
We  set $f_1(x_1,y)=f_1(x_3,y)=1$,  $f_1(x_4,y)=2$ and $f_1(u,v)=0$ for every $(u,v)\notin \{(x_1,y),(x_3,y),(x_4,y)\}$. 

\item$\alpha_1\geq 2\text{ and }\alpha_4=0$.
We   set $f_1(x_3,y)=f_1(x_4,y)=2$ and $f_1(u,v)=0$ for every $(u,v)\notin \{(x_3,y),(x_4,y)\}$.  

\item  $\alpha_1=1\text{ and }\alpha_4=1$.
We  set $f_1(x_1,y)=f_1(x_2,y)=f_1(x_3,y)=f_1(x_4,y)=1$  and $f_1(u,v)=0$ for every $v\ne y$ and $u\notin \{x_1,x_2,x_3,x_4\}$.  

\item  $\alpha_1=1\text{ and }\alpha_4=0$.
We  set $f_1(x_2,y)=f_1(x_4,y)=1$,  $f_1(x_3,y)=2$  and $f_1(u,v)=0$ for every $(u,v)\notin \{(x_2,y),(x_3,y),(x_4,y)\}$.  

\item  $\alpha_1=0\text{ and }\alpha_4=0$. 
We  set $f_1(x_2,y)=f_1(x_3,y)=2$  and $f_1(u,v)=0$ for every $(u,v)\notin \{(x_2,y),(x_3,y)\}$. 
\end{enumerate}

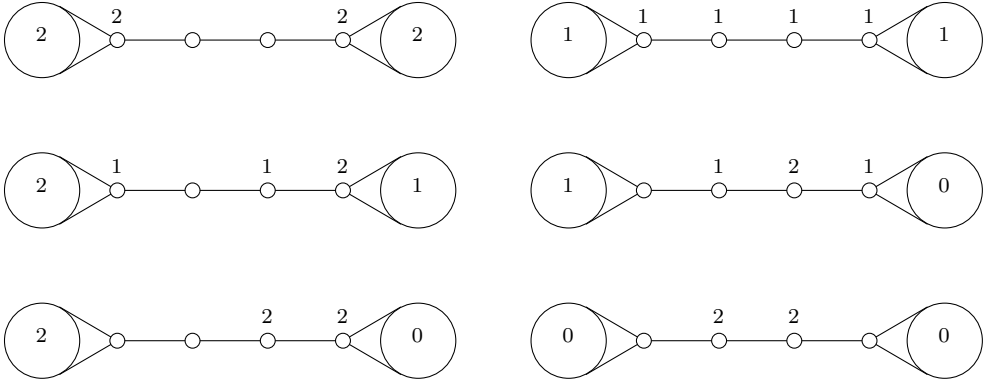
\begin{figure}[h]

\begin{tikzpicture}[transform shape, inner sep = .7mm]
\foreach \j in {2,4,6}
{
\foreach \i in {1,8}
{
\draw(\i+0.23,\j+0.45)--(\i+1,\j);
\draw (\i,\j) circle (0.5);
\draw(\i+0.23,\j-0.45)--(\i+1,\j);
}
\foreach \i in {6,13}
{
\draw(\i-0.23,\j+0.45)--(\i-1,\j);
\draw (\i,\j) circle (0.5);
\draw(\i-0.23,\j-0.45)--(\i-1,\j);
}
\draw(2,\j)--(3,\j)--(4,\j)--(5,\j);
\draw(9,\j)--(10,\j)--(11,\j)--(12,\j);
\foreach \i in {2,3,4,5,9,10,11,12}
{
\node [draw=black, shape=circle, fill=white] at  (\i,\j) {};
}
}

\node [] at (1,6) {$^{2}$};
\node [above] at (2,6) {$^{2}$};
\node [above] at (5,6) {$^{2}$};
\node [] at (6,6) {$^{2}$};

\node [] at (1,4) {$^{2}$};
\node [above] at (2,4) {$^{1}$};
\node [above] at (4,4) {$^{1}$};
\node [above] at (5,4) {$^{2}$};
\node [] at (6,4) {$^{1}$};

\node [] at (1,2) {$^{2}$};
\node [above] at (4,2) {$^{2}$};
\node [above] at (5,2) {$^{2}$};
\node [] at (6,2) {$^{0}$};

\node [] at (8,6) {$^{1}$};
\node [above] at (9,6) {$^{1}$};
\node [above] at (10,6) {$^{1}$};
\node [above] at (11,6) {$^{1}$};
\node [above] at (12,6) {$^{1}$};
\node [] at (13,6) {$^{1}$};

\node [] at (8,4) {$^{1}$};
\node [above] at (10,4) {$^{1}$};
\node [above] at (11,4) {$^{2}$};
\node [above] at (12,4) {$^{1}$};
\node [] at (13,4) {$^{0}$};

\node [] at (8,2) {$^{0}$};
\node [above] at (10,2) {$^{2}$};
\node [above] at (11,2) {$^{2}$};
\node [] at (13,2) {$^{0}$};
\end{tikzpicture}
\caption{The six cases above are described in this scheme.
}
\end{figure}
 A simple case analysis shows that the vertices of $G\circ H$ are defended  by the  assignment of legions produced by $f_1$. 
 Therefore, 
\begin{equation}\label{GHleG*H+4}
\gamma_r(G\circ H)\leq w(f_1)\leq w(f)+4= \gamma_r(G^*\circ H)+4.
\end{equation}

Now we will show that the equality holds. Let $g$ be a $\gamma_r(G\circ H)$-function satisfying Lemma \ref{twoouterweights}. We will construct a WRDF $g_1$ on $G^*\circ H$ from the function $g$ such that $w(g_1)\le w(g)-4$.
 We also need to define $B_1$ and $B_4$ as follows.
$$B_1=\bigcup_{\begin{array}{c}
 x\in N(x_1)\setminus \{x_2\}\\
\end{array} } V(H_x)\;\;\text{ and }\;\;B_4=\bigcup_{\begin{array}{c}
 x\in N(x_4)\setminus\{x_3\}\\
\end{array}} V(H_x) .$$
We define $g_1$ according to the following six cases:

\begin{enumerate}[1'.]
\item $g(B_1)\geq 2$ and $g(B_4)\geq 2$. In this case, we set $g_1(x,y)=g(x,y)$ for every $(x,y)\in V(G^*\circ H)$. 

\item $g(B_1)\geq 2$ and $g(B_4)=1$. Depending on $g(H_{1})$ we will consider the following two cases:

\begin{enumerate}[2'.1]
\item $g(H_{1})\leq 1$.
Since $g(H_{1})\leq g(B_4)$, we can  set $g_1(x,y)=g(x,y)$ for for every $(x,y)\in V(G^*\circ H)$.

\item $g(H_{1})\geq 2$. We will show that $g(S\times V(H))\geq 5$. To see this, we will try to place four  legions   in $S\times V(H)$ as shown in Figure \ref{GtoG*case2}, where $ a+b= 2$ (in Figures \ref{GtoG*case2}-\ref{GtoG*case4c} black vertices represent a contradiction with Lemma \ref{twoouterweights}).
Since  in all these cases we have a contradiction, we can conclude that $g(S\times V(H))\geq 5$.
Hence, we place the legions in the following way: for some $(x_0,y_0)\in B_4$ we set $g_1(x_0,y_0)=g(x_0,y_0)+1$ and $g_1(x,y)=g(x,y)$ for every $(x,y)\in V(G^*\circ H)\setminus \{(x_0,y_0)\}$.

\begin{figure}[!h]
\centering
\begin{tikzpicture}
[transform shape, inner sep = .7mm]

\foreach \j in {-1,1}
{
\foreach \i in {0,7}
{
\draw (\i+1,\j)--(\i+2,\j)--(\i+3,\j)--(\i+4,\j);
}
\foreach \i in {5,12}
{
\draw(\i-0.23,\j+0.45)--(\i-1,\j);
\draw (\i,\j) circle (0.5);
\draw(\i-0.23,\j-0.45)--(\i-1,\j);
}
\foreach \i in {0,7}
{
\draw(\i+0.23,\j+0.45)--(\i+1,\j);
\draw (\i,\j) circle (0.5);
\draw(\i+0.23,\j-0.45)--(\i+1,\j);
}
\foreach \i in {1,2,3,4,8,9,10,11}
{
\node [draw=black, shape=circle, fill=white] at  (\i,\j) {};
}
}

\node [] at (0,1) {$^{2}$};
\node [above] at (1,1) {$^{2}$};
\node [above] at (2,1) {$^{a}$};
\node [above] at (3,1) {$^{0}$};
\node [above] at (4,1) {$^{b}$};
\node [] at (5,1) {$^{1}$};
\node [draw=black, shape=circle, fill=black] at  (4,1) {};

\node [] at (7,1) {$^{2}$};
\node [above] at (8,1) {$^{2}$};
\node [above] at (9,1) {$^{0}$};
\node [above] at (10,1) {$^{2}$};
\node [above] at (11,1) {$^{0}$};
\node [] at (12,1) {$^{1}$};
\node [draw=black, shape=circle, fill=black] at  (10,1) {};

\node [] at (0,-1) {$^{2}$};
\node [above] at (1,-1) {$^{2}$};
\node [above] at (2,-1) {$^{1}$};
\node [above] at (3,-1) {$^{1}$};
\node [above] at (4,-1) {$^{0}$};
\node [] at (5,-1) {$^{1}$};
\node [draw=black, shape=circle, fill=black] at  (3,-1) {};

\node [] at (7,-1) {$^{2}$};
\node [above] at (8,-1) {$^{2}$};
\node [above] at (9,-1) {$^{0}$};
\node [above] at (10,-1) {$^{1}$};
\node [above] at (11,-1) {$^{1}$};
\node [] at (12,-1) {$^{1}$};
\node [draw=black, shape=circle, fill=black] at  (10,-1) {};

\end{tikzpicture}\caption{Scheme corresponding to Case 2'.2. }\label{GtoG*case2}
\end{figure}
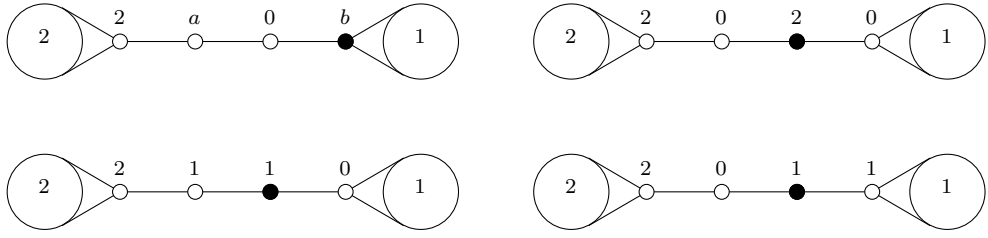 
\end{enumerate}

\item $g(B_1)\geq 2$ and $g(B_4)=0$.  In this case, we consider the following three cases depending on the value of $g(H_{1})$:

\begin{enumerate}[3'.1]
\item $g(H_{1 })=0$. 
In this case $g(H_{1})=g(B_4)$ so we set $g_1(x,y)=g(x,y)$ for every $(x,y)\in V(G^*\circ H)$.

\item $g(H_{1})=1$.
We will show that $g(S \times V(H))\geq 5$. To see this, we will try to place four legions in $S\times V(H)$ as shown in Figure \ref{GtoG*case3b}, where $2\le a+b\le 3$ and $c+d=1$. In both cases we have a contradiction with Lemma \ref{twoouterweights}. Hence, $g(S \times V(H))\geq 5$ and so we place the legions in the following way: for some $(x_0,y_0)\in B_4$ we set $g_1(x_0,y_0)=1$ and $g_1(x,y)=g(x,y)$ for every $(x,y)\in V(G^*\circ H)\setminus \{(x_0,y_0)\}$.

\begin{figure}[!ht]
\centering
\begin{tikzpicture}
[transform shape, inner sep = .7mm]

\foreach \j in {1}
{
\foreach \i in {0,7}
{
\draw (\i+1,\j)--(\i+2,\j)--(\i+3,\j)--(\i+4,\j);
}
\foreach \i in {5,12}
{
\draw(\i-0.23,\j+0.45)--(\i-1,\j);
\draw (\i,\j) circle (0.5);
\draw(\i-0.23,\j-0.45)--(\i-1,\j);
}
\foreach \i in {0,7}
{
\draw(\i+0.23,\j+0.45)--(\i+1,\j);
\draw (\i,\j) circle (0.5);
\draw(\i+0.23,\j-0.45)--(\i+1,\j);
}
\foreach \i in {1,2,3,4,8,9,10,11}
{
\node [draw=black, shape=circle, fill=white] at  (\i,\j) {};
}
}

\node [] at (0,1) {$^{2}$};
\node [above] at (1,1) {$^{1}$};
\node [above] at (2,1) {$^{a}$};
\node [above] at (3,1) {$^{\leq 1}$};
\node [above] at (4,1) {$^{b}$};
\node [] at (5,1) {$^{0}$};
\node [draw=black, shape=circle, fill=black] at  (4,1) {};

\node [] at (7,1) {$^{2}$};
\node [above] at (8,1) {$^{1}$};
\node [above] at (9,1) {$^{c}$};
\node [above] at (10,1) {$^{2}$};
\node [above] at (11,1) {$^{d}$};
\node [] at (12,1) {$^{0}$};
\node [draw=black, shape=circle, fill=black] at  (10,1) {};

\end{tikzpicture}\caption{Scheme corresponding to Case 3'.2.}\label{GtoG*case3b}
\end{figure}
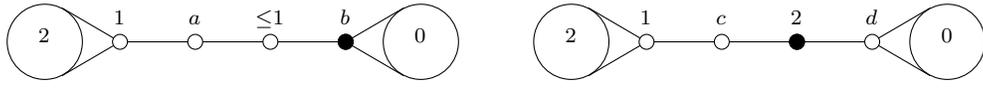

\item $g(H_{1})\geq 2$.
We will show that $g(S\times V(H))\geq 6$. To see this, we will try to place five legions in $S\times V(H)$ as shown in Figure \ref{GtoG*case3c}, where $2\le a+b\le 3$ and $c+d=1$. 
In both cases we have a contradiction with Lemma \ref{twoouterweights}. Hence,  $g(S \times V(H))\geq 6$ and so we place the legions in the following way: for some $(x_0,y_0)\in B_4$ we set $g_1(x_0,y_0)=2$ and $g_1(x,y)=g(x,y)$ for every $(x,y)\in V(G^*\circ H)\setminus \{(x_0,y_0)\}$.

\begin{figure}[!ht]
\centering
\begin{tikzpicture}
[transform shape, inner sep = .7mm]

\foreach \j in {1}
{
\foreach \i in {0,7}
{
\draw (\i+1,\j)--(\i+2,\j)--(\i+3,\j)--(\i+4,\j);
}
\foreach \i in {5,12}
{
\draw(\i-0.23,\j+0.45)--(\i-1,\j);
\draw (\i,\j) circle (0.5);
\draw(\i-0.23,\j-0.45)--(\i-1,\j);
}
\foreach \i in {0,7}
{
\draw(\i+0.23,\j+0.45)--(\i+1,\j);
\draw (\i,\j) circle (0.5);
\draw(\i+0.23,\j-0.45)--(\i+1,\j);
}
\foreach \i in {1,2,3,4,8,9,10,11}
{
\node [draw=black, shape=circle, fill=white] at  (\i,\j) {};
}
}

\node [] at (0,1) {$^{2}$};
\node [above] at (1,1) {$^{2}$};
\node [above] at (2,1) {$^{a}$};
\node [above] at (3,1) {$^{\leq 1}$};
\node [above] at (4,1) {$^{b}$};
\node [] at (5,1) {$^{0}$};
\node [draw=black, shape=circle, fill=black] at  (4,1) {};

\node [] at (7,1) {$^{2}$};
\node [above] at (8,1) {$^{2}$};
\node [above] at (9,1) {$^{c}$};
\node [above] at (10,1) {$^{2}$};
\node [above] at (11,1) {$^{d}$};
\node [] at (12,1) {$^{0}$};
\node [draw=black, shape=circle, fill=black] at  (10,1) {};

\end{tikzpicture}\caption{Scheme corresponding to Case 3'.3. }\label{GtoG*case3c}
\end{figure}
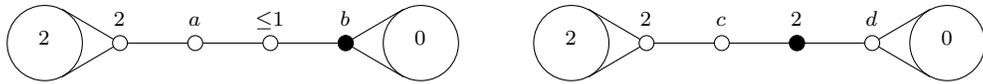
\end{enumerate}

\item $g(B_1)=g(B_4)=1$.
In this case, we consider the following three cases depending on the value of $g(H_{1})$ and $g(H_4)$:

\begin{enumerate}[4'.1]
\item $g(H_{1})\leq 1$ and $g(H_{4})\leq 1$.
In this case $g(H_{1})\leq g(B_4)$ and $g(H_{4})\leq g(B_1)$, so we set $g_1(x,y)=g(x,y)$ for every $(x,y)\in V(G^*\circ H)$.

\item $g(H_{1})\geq 2$ and $g(H_{4})\leq 1$ (this case is symmetric to $g(H_{1})\leq 1$ and $g(H_{4})\geq 2$). 
We will show that $g(S \times V(H))\geq 5$. To see this, we will try to place four legions in $S\times V(H)$ as shown in Figure \ref{GtoG*case4b}, where $a+b=1$. In all cases we have a contradiction with Lemma \ref{twoouterweights}. 
Hence, $g(S \times V(H))\geq 5$ and so we define $g_1$ as follows: for some $(x_0,y_0)\in B_4$ we set $g_1(x_0,y_0)=g(x_0,y_0)+1$ and $g_1(x,y)=g(x,y)$ for every $(x,y)\in V(G^*\circ H)\setminus \{(x_0,y_0)\}$.
\begin{figure}[!ht]
\centering
\begin{tikzpicture}
[transform shape, inner sep = .7mm]

\foreach \j in {-1,1}
{
\foreach \i in {0,7}
{
\draw (\i+1,\j)--(\i+2,\j)--(\i+3,\j)--(\i+4,\j);
}
\foreach \i in {5,12}
{
\draw(\i-0.23,\j+0.45)--(\i-1,\j);
\draw (\i,\j) circle (0.5);
\draw(\i-0.23,\j-0.45)--(\i-1,\j);
}
\foreach \i in {0,7}
{
\draw(\i+0.23,\j+0.45)--(\i+1,\j);
\draw (\i,\j) circle (0.5);
\draw(\i+0.23,\j-0.45)--(\i+1,\j);
}
\foreach \i in {1,2,3,4,8,9,10,11}
{
\node [draw=black, shape=circle, fill=white] at  (\i,\j) {};
}
}

\node [] at (0,1) {$^{1}$};
\node [above] at (1,1) {$^{2}$};
\node [above] at (2,1) {$^{a}$};
\node [above] at (3,1) {$^{b}$};
\node [above] at (4,1) {$^{1}$};
\node [] at (5,1) {$^{1}$};
\node [draw=black, shape=circle, fill=black] at  (1,1) {};
\node [draw=black, shape=circle, fill=black] at  (4,1) {};

\node [] at (7,1) {$^{1}$};
\node [above] at (8,1) {$^{2}$};
\node [above] at (9,1) {$^{2}$};
\node [above] at (10,1) {$^{0}$};
\node [above] at (11,1) {$^{0}$};
\node [] at (12,1) {$^{1}$};
\node [draw=black, shape=circle, fill=black] at  (11,1) {};

\node [] at (0,-1) {$^{1}$};
\node [above] at (1,-1) {$^{2}$};
\node [above] at (2,-1) {$^{0}$};
\node [above] at (3,-1) {$^{2}$};
\node [above] at (4,-1) {$^{0}$};
\node [] at (5,-1) {$^{1}$};
\node [draw=black, shape=circle, fill=black] at  (3,-1) {};

\node [] at (7,-1) {$^{1}$};
\node [above] at (8,-1) {$^{2}$};
\node [above] at (9,-1) {$^{1}$};
\node [above] at (10,-1) {$^{1}$};
\node [above] at (11,-1) {$^{0}$};
\node [] at (12,-1) {$^{1}$};
\node [draw=black, shape=circle, fill=black] at  (10,-1) {};

\end{tikzpicture}\caption{Scheme corresponding to Case 4'.2. }\label{GtoG*case4b}
\end{figure}
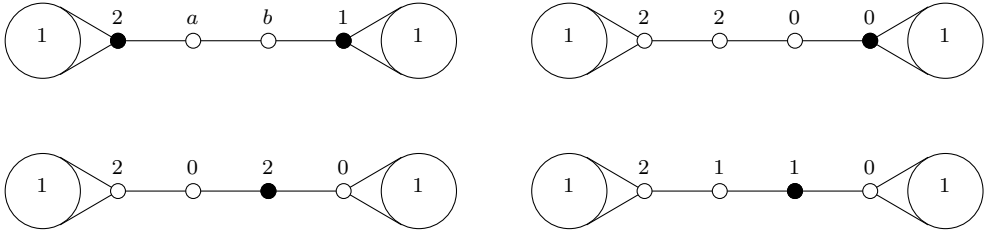

\item $g(H_{1})\geq 2$ and $g(H_{4})\geq 2$.

We will show that $g(S\times V(H))\geq 6$. To see this, we will try to place five legions in $S\times V(H)$ as shown in Figure \ref{GtoG*case4c}, where $a+b=1$. 
In this case we have a contradiction with Lemma \ref{twoouterweights}. Hence, $g(S \times V(H))\geq 6$ and so we  place the legions in the following way: for some $(x_0,y_0)\in B_1$ we set $g_1(x_0,y_0)=g(x_0,y_0)+1$, for some $(x_0',y_0')\in B_4$ we set $g_1(x_0',y_0')=g(x_0',y_0')+1$ and $g_1(x,y)=g(x,y)$ for every $(x,y)\in V(G^*\circ H)\setminus\{(x_0,y_0),(x_0',y_0')\}$.

\begin{figure}[!ht]
\centering
\begin{tikzpicture}
[transform shape, inner sep = .7mm]

\foreach \j in {1}
{
\foreach \i in {0}
{
\draw (\i+1,\j)--(\i+2,\j)--(\i+3,\j)--(\i+4,\j);
}
\foreach \i in {5}
{
\draw(\i-0.23,\j+0.45)--(\i-1,\j);
\draw (\i,\j) circle (0.5);
\draw(\i-0.23,\j-0.45)--(\i-1,\j);
}
\foreach \i in {0}
{
\draw(\i+0.23,\j+0.45)--(\i+1,\j);
\draw (\i,\j) circle (0.5);
\draw(\i+0.23,\j-0.45)--(\i+1,\j);
}
\foreach \i in {1,2,3,4}
{
\node [draw=black, shape=circle, fill=white] at  (\i,\j) {};
}
}

\node [] at (0,1) {$^{1}$};
\node [above] at (1,1) {$^{2}$};
\node [above] at (2,1) {$^{a}$};
\node [above] at (3,1) {$^{b}$};
\node [above] at (4,1) {$^{2}$};
\node [] at (5,1) {$^{1}$};
\node [draw=black, shape=circle, fill=black] at  (1,1) {};
\node [draw=black, shape=circle, fill=black] at  (4,1) {};

\end{tikzpicture}\caption{Scheme corresponding to Case 4'.3. }\label{GtoG*case4c}
\end{figure}
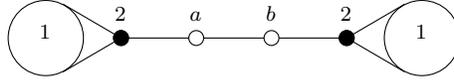
\end{enumerate}

\item $g(B_1)=1$ and $g(B_4)=0$. Notice that if $g(H_2)=0$ or $g(H_3)\le 1$, then we have a contradiction with Lemma \ref{twoouterweights}, so that $g(H_2)\ge 1$ and $g(H_3)\ge 2$. We differentiate two cases according to the value of $g(H_2)$:

\begin{enumerate}[5'.1]
\item $g(H_{2})=1$. By Lemma  \ref{twoouterweights}, we have that $g(H_{4})\ge 1$. Thus, we place the legions in the following way:  for some $(x_0,y_0)\in B_1$ we set $g_1(x_0,y_0)=\min\{2,g(H_4)-1\}$, for some $(x'_0,y'_0)\in B_4$ we set $g_1(x'_0,y'_0)=\min\{2,g(H_1)\}$ and $g_1(x,y)=g(x,y)$ for every $(x,y)\in V(G^*\circ H)\setminus \{(x_0,y_0), (x'_0,y'_0)\}$.

\item $g(H_{2})\ge 2$. In this case, we  place the legions in the following way: for some $(x_0,y_0)\in B_1$ we set $g_1(x_0,y_0)=\min\{2,g(H_4)\}$, for some $(x'_0,y'_0)\in B_4$ we set $g_1(x'_0,y'_0)=\min\{2,g(H_1)\}$ and $g_1(x,y)=g(x,y)$ for every $(x,y)\in V(G^*\circ H)\setminus \{(x_0,y_0), (x'_0,y'_0)\}$.
\end{enumerate}

\item $g(B_1)=g(B_4)=0$. Notice that if $g(H_2)\le 1$ or $g(H_3)\le 1$, then we have a contradiction with Lemma \ref{twoouterweights}, so that $g(H_2)\ge 2$ and $g(H_3)\ge 2$. We place the legions in the following way: for some $(x_0,y_0)\in B_1$ we set $g_1(x_0,y_0)=\min\{2,g(H_4)\}$, for some $(x'_0,y'_0)\in B_4$ we set $g_1(x'_0,y'_0)=\min\{2,g(H_1)\}$ and $g_1(x,y)=g(x,y)$ for every $(x,y)\in V(G^*\circ H)\setminus \{(x_0,y_0), (x'_0,y'_0)\}$.
\end{enumerate}

A simple case analysis shows that the vertices of $G^*\circ H$ are defended by the assignment of legions produced by $g_1$. 
Therefore, 
\begin{equation}\label{G*HleGH-4}
\gamma_r(G^*\circ H)\leq w(g_1)\le w(g)-4\leq\gamma_r(G\circ H) -4.
\end{equation}
Finally, by \eqref{GHleG*H+4} and \eqref{G*HleGH-4} we can conclude that
$\gamma_r(G\circ H)=\gamma_r(G^*\circ H)+4,$ as claimed.
\end{proof}

\section{Open problems}
\label{SectionOpenProblems}
Some closed formulae for $\gamma_r(G\circ H)$, obtained in Section \ref{SectionClosedFormulae}, were derived under the assumption  that  $\gamma_t(G)=\frac{1}{2}\max\{\gamma_r(G),2\rho(G)\}$ or $\gamma_{2,t}(G)=\max\{\gamma_r(G),2\rho(G)\}$ or $\gamma_r(G)=2\gamma(G)$.
This suggests the following open problems.

\begin{problem}\label{ProblemWRD=2dom}
Characterize the graphs with $\gamma_r(G)=2\gamma(G)$.
\end{problem}

\begin{problem}\label{ProblemWRD=2totaldom}
Characterize the graphs with $\gamma_r(G)=2\gamma_t(G)$.
\end{problem}

\begin{problem}
Characterize the graphs with $\gamma_{t}(G)=\rho(G)$.
\end{problem}

\begin{problem}
Characterize the graphs with $\gamma_r(G)=\gamma_{2,t}(G)$.
\end{problem}

\begin{problem}
Characterize the graphs with $\gamma_{2,t}(G)=2\rho(G)$.
\end{problem}

Notice that $\gamma_r(G)\le \gamma_R(G)\le 2\gamma(G)\le 2\gamma_t(G)$. Hence,  $\gamma_r(G)=2\gamma(G)$ if and only if $\gamma_r(G)=\gamma_R(G)$ and $G$ is a Roman graph. Furthermore,   
   $\gamma_r(G)=2\gamma_t(G)$ if and only if  all equalities hold true in the previous domination chain. Therefore, the starting point to solve Problems \ref{ProblemWRD=2dom} and \ref{ProblemWRD=2totaldom} is a deep investigation of Roman graphs. 

\section*{Acknowledgements}
Research supported in part by  the Spanish  government  under  the grant MTM2016-78227-C2-1-P.

\bibliographystyle{elsart-num-sort}



\end{document}